\documentclass[reqno]{amsart}

\usepackage{amsmath,amsfonts,amssymb,amscd,verbatim,delarray,fullpage}

\usepackage{tikz}
\usepackage{tikz-cd}
\usetikzlibrary{matrix,arrows}
\usepackage{stmaryrd}
\usepackage{amsthm}
\usepackage{url}
\usepackage[polutonikogreek,english]{babel}

\DeclareSymbolFont{bbold}{U}{bbold}{m}{n}
\DeclareSymbolFontAlphabet{\mathbbold}{bbold}

\DeclareMathOperator{\JH}{JH}
\DeclareMathOperator{\weier}{Weier}

\DeclareMathOperator{\occ}{Occ}

\DeclareMathOperator{\cyc}{cyc}

\DeclareMathOperator{\res}{res}

\DeclareMathOperator{\ns}{ns}
\DeclareMathOperator{\s}{s}
\DeclareMathOperator{\aut}{Aut}

\DeclareMathOperator{\id}{id}

\DeclareMathOperator{\tors}{tors}
\DeclareMathOperator{\GL}{GL}
\DeclareMathOperator{\GSp}{GSp}
\DeclareMathOperator{\SL}{SL}
\DeclareMathOperator{\PGL}{PGL}
\DeclareMathOperator{\PSL}{PSL}

\DeclareMathOperator{\tr}{tr}

\DeclareMathOperator{\rad}{rad}
\DeclareMathOperator{\pr}{pr}

\chardef\bslash=`\\ 

\begin{document}


\newtheorem{Theorem}{Theorem}[section]

\newtheorem{cor}[Theorem]{Corollary}
\newtheorem{goal}[Theorem]{Goal}

\newtheorem{Conjecture}[Theorem]{Conjecture}
\newtheorem{guess}[Theorem]{Guess}

\newtheorem{exercise}[Theorem]{Exercise}
\newtheorem{Question}[Theorem]{Question}
\newtheorem{lemma}[Theorem]{Lemma}
\newtheorem{property}[Theorem]{Property}
\newtheorem{proposition}[Theorem]{Proposition}
\newtheorem{ax}[Theorem]{Axiom}
\newtheorem{claim}[Theorem]{Claim}

\newtheorem{nTheorem}{Surjectivity Theorem}

\theoremstyle{definition}
\newtheorem{Definition}[Theorem]{Definition}
\newtheorem{problem}[Theorem]{Problem}
\newtheorem{question}[Theorem]{Question}
\newtheorem{Example}[Theorem]{Example}

\newtheorem{remark}[Theorem]{Remark}
\newtheorem{diagram}{Diagram}
\newtheorem{Remark}[Theorem]{Remark}
\newcommand{\diagref}[1]{diagram~\ref{#1}}
\newcommand{\thmref}[1]{Theorem~\ref{#1}}
\newcommand{\secref}[1]{Section~\ref{#1}}
\newcommand{\subsecref}[1]{Subsection~\ref{#1}}
\newcommand{\lemref}[1]{Lemma~\ref{#1}}
\newcommand{\corref}[1]{Corollary~\ref{#1}}
\newcommand{\exampref}[1]{Example~\ref{#1}}
\newcommand{\remarkref}[1]{Remark~\ref{#1}}
\newcommand{\corlref}[1]{Corollary~\ref{#1}}
\newcommand{\claimref}[1]{Claim~\ref{#1}}
\newcommand{\defnref}[1]{Definition~\ref{#1}}
\newcommand{\propref}[1]{Proposition~\ref{#1}}
\newcommand{\prref}[1]{Property~\ref{#1}}
\newcommand{\itemref}[1]{(\ref{#1})}
\newcommand{\ul}[1]{\underline{#1}}


\newcommand{\CE}{\mathcal{E}}
\newcommand{\CG}{\mathcal{G}}\newcommand{\CV}{\mathcal{V}}
\newcommand{\CL}{\mathcal{L}}
\newcommand{\CM}{\mathcal{M}}
\newcommand{\A}{\mathcal{A}}
\newcommand{\CO}{\mathcal{O}}
\newcommand{\B}{\mathcal{B}}
\newcommand{\CS}{\mathcal{S}}
\newcommand{\CX}{\mathcal{X}}
\newcommand{\CY}{\mathcal{Y}}
\newcommand{\CT}{\mathcal{T}}
\newcommand{\CW}{\mathcal{W}}
\newcommand{\CJ}{\mathcal{J}}
\newcommand{\Q}{\mathbb{Q}}
\newcommand{\Z}{\mathbb{Z}}

\newcommand{\st}{\sigma}
\renewcommand{\k}{\varkappa}
\newcommand{\Frac}{\mbox{Frac}}
\newcommand{\XC}{\mathcal{X}}
\newcommand{\wt}{\widetilde}
\newcommand{\wh}{\widehat}
\newcommand{\mk}{\medskip}
\renewcommand{\sectionmark}[1]{}
\renewcommand{\Im}{\operatorname{Im}}
\renewcommand{\Re}{\operatorname{Re}}
\newcommand{\la}{\langle}
\newcommand{\ra}{\rangle}
\newcommand{\LND}{\mbox{LND}}
\newcommand{\Pic}{\mbox{Pic}}
\newcommand{\lnd}{\mbox{lnd}}
\newcommand{\GLND}{\mbox{GLND}}\newcommand{\glnd}{\mbox{glnd}}
\newcommand{\Der}{\mbox{DER}}\newcommand{\DER}{\mbox{DER}}
\renewcommand{\th}{\theta}
\newcommand{\ve}{\varepsilon}
\newcommand{\1}{^{-1}}
\newcommand{\iy}{\infty}
\newcommand{\iintl}{\iint\limits}
\newcommand{\capl}{\operatornamewithlimits{\bigcap}\limits}
\newcommand{\cupl}{\operatornamewithlimits{\bigcup}\limits}
\newcommand{\suml}{\sum\limits}
\newcommand{\ord}{\operatorname{ord}}
\newcommand{\gal}{\operatorname{Gal}}
\newcommand{\bk}{\bigskip}
\newcommand{\fc}{\frac}
\newcommand{\g}{\gamma}
\newcommand{\be}{\beta}
\newcommand{\dl}{\delta}
\newcommand{\Dl}{\Delta}
\newcommand{\lm}{\lambda}
\newcommand{\Lm}{\Lambda}
\newcommand{\om}{\omega}
\newcommand{\ov}{\overline}
\newcommand{\vp}{\varphi}
\newcommand{\kap}{\varkappa}

\newcommand{\Vp}{\Phi}
\newcommand{\Varphi}{\Phi}
\newcommand{\BC}{\mathbb{C}}
\newcommand{\C}{\mathbb{C}}\newcommand{\BP}{\mathbb{P}}
\newcommand{\BQ}{\mathbb {Q}}
\newcommand{\BM}{\mathbb{M}}
\newcommand{\mbh}{\mathbb{H}}
\newcommand{\BR}{\mathbb{R}}\newcommand{\BN}{\mathbb{N}}
\newcommand{\BZ}{\mathbb{Z}}\newcommand{\BF}{\mathbb{F}}
\newcommand{\BA}{\mathbb {A}}
\renewcommand{\Im}{\operatorname{Im}}
\newcommand{\idd}{\operatorname{id}}
\newcommand{\ep}{\epsilon}
\newcommand{\tp}{\tilde\partial}
\newcommand{\doe}{\overset{\text{def}}{=}}
\newcommand{\supp} {\operatorname{supp}}
\newcommand{\loc} {\operatorname{loc}}
\newcommand{\de}{\partial}
\newcommand{\z}{\zeta}
\renewcommand{\a}{\alpha}
\newcommand{\G}{\Gamma}
\newcommand{\der}{\mbox{DER}}

\newcommand{\Spec}{\operatorname{Spec}}
\newcommand{\Sym}{\operatorname{Sym}}
\newcommand{\Aut}{\operatorname{Aut}}

\newcommand{\Idd}{\operatorname{Id}}

\newcommand{\tG}{\widetilde G}

\newcommand{\FX}{\mathfrac {X}}
\newcommand{\FV}{\mathfrac {V}}
\newcommand{\SX}{\mathcal {X}}
\newcommand{\SV}{\mathcal {V}}
\newcommand{\SO}{\mathcal {O}}
\newcommand{\SD}{\mathcal {D}}
\newcommand{\Sr}{\rho}
\newcommand{\SR}{\mathcal {R}}
\newcommand{\cl}{\mathcal{C}}
\newcommand{\ok}{\mathcal{O}_K}
\newcommand{\ab}{\mathcal{AB}}

\setcounter{equation}{0} \setcounter{section}{0}

\newcommand{\ds}{\displaystyle}
\newcommand{\gl}{\lambda}
\newcommand{\gL}{\Lambda}
\newcommand{\gge}{\epsilon}
\newcommand{\gG}{\Gamma}
\newcommand{\ga}{\alpha}
\newcommand{\gb}{\beta}
\newcommand{\gd}{\delta}
\newcommand{\gD}{\Delta}
\newcommand{\gs}{\sigma}
\newcommand{\mbq}{\mathbb{Q}}
\newcommand{\mbr}{\mathbb{R}}
\newcommand{\mbz}{\mathbb{Z}}
\newcommand{\mbc}{\mathbb{C}}
\newcommand{\mbn}{\mathbb{N}}
\newcommand{\mbp}{\mathbb{P}}
\newcommand{\mbf}{\mathbb{F}}
\newcommand{\mbe}{\mathbb{E}}
\newcommand{\lcm}{\text{lcm}\,}
\newcommand{\mf}[1]{\mathfrak{#1}}
\newcommand{\ol}[1]{\overline{#1}}
\newcommand{\mc}[1]{\mathcal{#1}}
\newcommand{\nequiv}{\equiv\hspace{-.07in}/\;}
\newcommand{\bnequiv}{\equiv\hspace{-.13in}/\;}

\title{A bound for the conductor of an open subgroup of $GL_2$ associated to an elliptic curve}
\author{Nathan Jones}
\address{University of Illinois at Chicago, 851 S Morgan St, SEO 322, Chicago, IL 60614}
\email{ncjones@uic.edu}

\renewcommand{\thefootnote}{\fnsymbol{footnote}} 
\renewcommand{\thefootnote}{\arabic{footnote}}

\renewcommand{\thefootnote}{\fnsymbol{footnote}} 
\renewcommand{\thefootnote}{\arabic{footnote}}


\date{}

\begin{abstract}
Given an elliptic curve $E$ without complex multiplication defined over a number field $K$, consider the image of the Galois representation defined by letting Galois act on the torsion of $E$.  Serre's open image theorem implies that there is a positive integer $m$ for which the Galois image is completely determined by its reduction modulo $m$.  In this note, we prove a bound on the smallest such $m$ in terms of standard invariants associated with $E$. The bound is sharp and improves upon previous results.
\end{abstract}

\maketitle

\section{introduction}

Let $K$ be a number field, let $E/K$ be an elliptic curve and let $E_{\tors}$ denote its torsion subgroup.  Denote by $G_K := \gal(\ol{K}/K)$ the absolute Galois group of $K$ and consider the Galois representation
\[
\rho_{E,K} : G_K \longrightarrow \aut(E_{\tors}) \simeq \GL_2(\hat{\mbz}) 
\]
defined by letting $G_K$ act on the torsion of $E$ and choosing a $\hat{\mbz}$-basis thereof.
A celebrated theorem of J.-P. Serre \cite{serre} states that, if $E$ has no complex multiplication, then the image of $\rho_{E,K}$ is open inside $\GL_2(\hat{\mbz})$, or equivalently that
\begin{equation} \label{serresopenimage}
\left[ \GL_2(\hat{\mbz}) : \rho_{E,K}(G_K) \right] < \infty.
\end{equation}
Consequently, one may find a positive integer $m$ with the property that
\begin{equation*} 
\ker\left( \GL_2(\hat{\mbz}) \rightarrow \GL_2(\mbz/m\mbz) \right) \subseteq \rho_{E,K}(G_K).
\end{equation*}
\begin{Definition} \label{definitionofme}
Given an open subgroup $G \subseteq \GL_2(\hat{\mbz})$, we define the positive integer $m_G$ by
\[
m_G := \min \{ m \in \mbn : \ker\left( \GL_2(\hat{\mbz}) \rightarrow \GL_2(\mbz/m\mbz) \right) \subseteq G \}
\]
and call it the {\bf{conductor of $G$}}.  In case $G = \rho_{E,K}(G_K)$ for an elliptic curve $E$ defined over a number field $K$ and without complex multiplication, we denote the conductor of $G$ by $m_{E,K}$.
\end{Definition}
The purpose of this note is to prove the following upper bound for $m_{E,K}$.  In its statement, $\gD_K$ denotes the absolute discriminant of the number field $K$, $\gD_E$ denotes the minimal discriminant ideal attached to the elliptic curve $E$, $N_{K/\mbq} : K^\times \longrightarrow \mbq^\times$ denotes the usual norm map and 
\[
\rad(m) := \prod_{\ell \mid m \atop \ell \text{ prime}} \ell
\]
denotes the radical of the positive integer $m$.  Given a non-zero ideal $I \subseteq \mc{O}_K$, we identify the ideal $N_{K/\mbq}(I) \subseteq \mbz$ with the (unique) positive integer that generates it, and thus we may regard $N_{K/\mbq}(\gD_E) \in \mbn$.
\begin{Theorem} \label{maintheorem}
Let $K$ be a number field, let $E$ be an elliptic curve over $K$ without complex multiplication, and let $m_{E,K} \in \mbn$ be as in Definition \ref{definitionofme}.  Then one has
\[
m_{E,K} \leq 2 \cdot \left[ \GL_2(\hat{\mbz}) : \rho_{E,K}(G_K) \right] \cdot \rad\left( |\gD_K N_{K/\mbq}(\gD_E) | \right).
\]
\end{Theorem}
\begin{remark}
The bound in Theorem \ref{maintheorem} both improves upon and generalizes a bound appearing in \cite{jones1} (see Corollary \ref{conditionalcorollary} below).  Furthermore, using results in \cite{daniels}, we may see that there are infinitely many\footnote{Specifically, \eqref{mEforsomeserrecurves} holds for any Serre curve $E$ with the property that $\gD_E$ is square-free and $\gD_E \not\equiv 1 \mod 4$.} elliptic curves $E$ over $\mbq$ satisfying 
\begin{equation} \label{mEforsomeserrecurves}
m_{E,\mbq} = 2 \cdot \left[ \GL_2(\hat{\mbz}) : \rho_{E,\mbq}(G_\mbq) \right] \cdot \rad(| \gD_E | ).
\end{equation}
Thus, our bound for $m_{E,K}$ is sharp when $K = \mbq$.
\end{remark}
\begin{remark}
Let 
\[
\begin{split}
\rho_{E,m} : G_K &\longrightarrow \GL_2(\mbz/m\mbz), \\
\rho_{E,\ell^\infty} : G_K &\longrightarrow \GL_2(\mbz_\ell)
\end{split}
\]
denote the Galois representations defined by letting $G_K$ act on $E[m]$ and on $\ds E[\ell^\infty] := \bigcup_{n \geq 1} E[\ell^n]$ respectively, and let $K(E[m]) = \ol{K}^{\ker \rho_{E,m}(G_K)}$ denote the $m$-th division field of $E$.
The conductor $m_{E,K}$ that we are considering should not be confused with ``Serre's constant,'' defined for an elliptic curve $E$ over $\mbq$ in \cite{danielsgonzalez} (see also \cite{cojocaru}) by
\[
A(E) := \prod_{{\begin{substack} { \ell^n \text{ a prime power} \\ \rho_{E,\ell^n}(G_\mbq) \neq \GL_2(\mbz/\ell^n\mbz) \\ \forall k < n, \, \rho_{E,\ell^{k}}(G_\mbq) = \GL_2(\mbz/\ell^{k}\mbz) } \end{substack}}} \ell^n.
\]
It is evident that $A(E)$ divides $m_{E,\mbq}$, but $m_{E,\mbq}$ is in general larger than $A(E)$.  The main differences between these two constants are as follows:
\begin{enumerate}
\item A prime power $\ell^n$ divides $m_{E,\mbq}$ whenever $\ker\left( \GL_2(\mbz_\ell) \rightarrow \GL_2(\mbz/\ell^{n-1}\mbz) \right) \not\subseteq \rho_{E,\ell^\infty}(G_\mbq)$, whereas $A(E)$ is square-free, except possibly at the primes $2$ and $3$.  In other words, for each prime $\ell$, $m_{E,\mbq}$ encodes the action of $G_\mbq$ on the entire $\ell$-adic Tate module, whereas, for $\ell \geq 5$, $A(E)$ only encodes the action of $G_\mbq$ on the $\ell$-torsion of $E$.
\item It may happen that there is a non-trivial intersection $\mbq \neq \mbq(E[m_1]) \cap \mbq(E[m_2])$ for some $m_1, m_2 \in \mbn$ with $\gcd(m_1,m_2) = 1$.  The constant $m_{E,\mbq}$ encodes such ``entanglements,'' whereas $A(E)$ does not.
\end{enumerate}
The general phenomenon of entanglements has come up in various recent papers, see for instance \cite{braujones}, which studies elliptic curves $E$ over $\mbq$ satisfying $\left[ \mbq(E[2]) : \mbq \right] = 6$ and $\mbq(E[2]) \subseteq \mbq(E[3])$, and also \cite{bourdon}, in which potential entanglements come up in an analysis of sporadic points on the modular curve $X_1(N)$.
\end{remark}

Given an elliptic curve $E$ defined over a number field $K$, computing the positive integer $m_{E,K}$ is a step toward understanding the image $\rho_{E,K}(G_K) \subseteq \GL_2(\hat{\mbz})$.  Following Serre's open image result, there has been much interest in the nature of $\rho_{E,K}(G_K)$, for instance regarding its mod $\ell$ reductions (see \cite{mazur}, \cite{merel}, \cite{lozanorobledo}, \cite{biluparent}, \cite{biluparentrebolledo}, and \cite{zywina2}) and also more recently its reductions at composite levels (see \cite{dokchitser}, \cite{sutherlandzywina}, \cite{danielsgonzalez} and \cite{morrow}).  In addition to this connection, Theorem \ref{maintheorem} also has analytic relevance; for instance in \cite{awstitchmarsh} it is applied to the study averages of constants appearing in various elliptic curve conjectures.

Serre's open image result \eqref{serresopenimage} implies that, for any $E/K$ without complex multiplication (CM), there exists a bound $C_{E,K} > 0$ so that, for each prime $\ell > C_{E,K}$, we have $\rho_{E,\ell}(G_K) = \GL_2(\mbz/\ell\mbz)$.
Serre asked whether the constant $C_{E,K}$ may be chosen uniformly in $E$, i.e. whether
\begin{equation} \label{serresquestioneqn}
\exists C_K > 0 \; \text{ so that, } \; \forall E/K \text{ without CM and } \forall \text{ prime } \ell > C_K, \; \rho_{E,\ell}(G_K) = \GL_2(\mbz/\ell\mbz).
\end{equation}
This question is still open, even in the case $K = \mbq$.  An affirmative answer to it would imply that
\[
[ \GL_2(\hat{\mbz}) : \rho_E(G_K) ] \ll_K 1,
\]
although the implied constant is ineffective, because of an appeal to Faltings' theorem (see \cite{zywinaindex}, which details this in the case $K = \mbq$).  Theorem \ref{maintheorem}  thus has the following corollary.
\begin{cor} \label{conditionalcorollary}
Assume that \eqref{serresquestioneqn} holds.  We then have
\[
m_{E,K} \ll_K \rad \left( N_{K/\mbq}(\gD_E) \right).
\]
\end{cor}

Theorem \ref{maintheorem} is proved via the following two propositions, the first of which deals generally with open subgroups $G \subseteq \GL_2(\hat{\mbz})$.
Because of group-theoretical differences present for the prime $2$ (see \cite{dokchitser} and \cite{elkies}, which concern the Galois representation on the $2$-adic (resp. on the $3$-adic) Tate module, illustrating these differences), it will be convenient to introduce the following modified radical:
\begin{equation} \label{defofradprime}
\rad'(m) := 
\begin{cases}
\rad(m) & \text{ if } 4 \nmid m \\
2\rad(m) & \text{ if } 4 \mid m.
\end{cases}
\end{equation}
We will also distinguish the following case involving the prime $3$, in whose statement $G_3$ (resp. $G(3))$ denotes the image of $G$ under the projection map $\GL_2(\hat{\mbz}) \longrightarrow \GL_2(\mbz_3)$ (resp. under $\GL_2(\hat{\mbz}) \longrightarrow \GL_2(\mbz/3\mbz)$).  The analysis proceeds a bit differently according to whether or not the condition
\begin{equation} \label{theconditionat3}
9 \mid m_G, \quad \SL_2(\mbz_3) \not \subseteq G_3 \quad \text{ and } \quad  G(3) = \GL_2(\mbz/3\mbz)
\end{equation}
holds.  
\begin{proposition} \label{mGprimeboundprop}
Let $G \subseteq \GL_2(\hat{\mbz})$ be an open subgroup and let $m_G$ be as in Definition \ref{definitionofme}.  We then have
\[
\frac{m_G}{\rad'(m_G)} \; \text{ divides } \; \left[ \pi^{-1}(G(\rad'(m_G))) : G(m_G) \right],
\]
where $\rad'(\cdot)$ is defined as in \eqref{defofradprime} and $\pi : \GL_2(\mbz/m_G\mbz) \longrightarrow \GL_2(\mbz/\rad'(m_G)\mbz)$ denotes the canonical projection map.  Assuming that \eqref{theconditionat3} holds, we have
\[
\frac{9m_G}{\rad'(m_G)} \; \text{ divides } \; \left[ \pi^{-1}(G(\rad'(m_G))) : G(m_G) \right].
\]
\end{proposition}
In contrast with Proposition \ref{mGprimeboundprop}, our second proposition is specific to the situation where $G = \rho_{E,K}(G_K)$, making use of facts about the Weil pairing on an elliptic curve, together with the Ner\'{o}n-Ogg-Shafarevich criterion for ramification in division fields.
\begin{proposition} \label{radmGdoubleprimeboundprop}
Let $K$ be a number field and let $E$ be an elliptic curve defined over $K$ without complex multiplication.  Let $G := \rho_{E,K}(G_K)$ be the image of the Galois representation associated to $E$ and let $m_G$ be as in Definition \ref{definitionofme}.  Assuming that \eqref{theconditionat3} does not hold, we have
\[
\rad'(m_G) \leq  2 \left[ \GL_2(\mbz/\rad'(m_G)\mbz) : G(\rad'(m_G)) \right] \rad(|\gD_K N_{K/\mbq}(\gD_E)|).
\]
If \eqref{theconditionat3} does hold, then
\[
\frac{\rad'(m_G)}{3} \leq  2 \left[ \GL_2(\mbz/\rad'(m_G)\mbz) : G(\rad'(m_G)) \right] \rad(|\gD_K N_{K/\mbq}(\gD_E)|).
\]
\end{proposition}
Since the index of a subgroup is preserved under taking the full pre-image, we have that
\[
\left[ \GL_2(\mbz/\rad'(m_G)\mbz) : G(\rad'(m_G)) \right] = \left[ \GL_2(\mbz/m_G\mbz) : \pi^{-1}(G(\rad'(m_G))) \right],
\]
where $\pi : \GL_2(\mbz/m_G\mbz) \longrightarrow \GL_2(\mbz/\rad'(m_G)\mbz)$ is the canonical projection map. Thus, Theorem \ref{maintheorem} follows from Propositions \ref{mGprimeboundprop} and \ref{radmGdoubleprimeboundprop}.

Many of the ingredients that enter into the proof of Theorem \ref{maintheorem} may be verified for algebraic groups other than $\GL_2$.  For instance, using these same techniques, one should be able to obtain a similar bound for the analogous integer $m_{A,K}$ associated to an abelian variety $A$ defined over a number field $K$ whose Galois representation has open image inside $\GSp_{2g}(\hat{\mbz})$.

\subsection*{Acknowledgements} I would like to thank J. Mayle for thoughtful comments on a previous version and also the anonymous referee for carefully reading the manuscript and giving several helpful suggestions.

\section{Notation and preliminaries}

Throughout the paper, $p$ and $\ell$ will always denote prime numbers.  As usual, $\mbn$ denotes the set of natural numbers (excluding zero) and $\mbz$ denotes the set of integers.  We will occasionally use the abbreviations 
\[
\begin{split}
\mbn_{\geq \ga} &:= \{ n \in \mbn : n \geq \ga \}, \\
\mbz_{\geq \ga} &:= \{ n \in \mbz : n \geq \ga \}.
\end{split}
\]
We recall that 
\[
\hat{\mbz} := \lim_{\leftarrow} \mbz/m\mbz
\]
is the inverse limit of the rings $\mbz/m\mbz$ with respect to the canonical projection maps $\mbz/nm\mbz \rightarrow \mbz/m\mbz$.  Under the isomorphism of the Chinese Remainder Theorem, we have that
\begin{equation} \label{CRTisomorphism}
\hat{\mbz} \simeq \prod_{\ell} \mbz_\ell,
\end{equation}
where $\ds \mbz_\ell$ as usual denotes the ring of $\ell$-adic integers.  More generally, for any $m \in \mbn_{\geq 2}$ we define $\mbz_m$ and $\mbz_{(m)}$ to be the quotients of $\hat{\mbz}$ corresponding under \eqref{CRTisomorphism} to the following rings:
\[
\mbz_{m} \simeq \prod_{\ell \mid m} \mbz_\ell, \quad\quad \mbz_{(m)} \simeq \prod_{\ell \nmid m} \mbz_\ell.
\]
For any $m \in \mbn_{\geq 2}$ we have an isomorphism
\[
\hat{\mbz} \simeq \mbz_m \times \mbz_{(m)},
\]
and projection maps
\[
\begin{split}
&\hat{\mbz} \longrightarrow \mbz_m, \\
&\hat{\mbz} \longrightarrow \mbz_{(m)}.
\end{split}
\]
We note that these observations may also be applied to points in an algebraic group; in particular we have
\[
\GL_2(\hat{\mbz}) \simeq \GL_2(\mbz_m) \times \GL_2(\mbz_{(m)}) \simeq \prod_{\ell} \GL_2(\mbz_\ell)
\]
and we have projection maps
\begin{equation} \label{canonicalprojections}
\begin{split}
\pi_m : &\GL_2(\hat{\mbz}) \longrightarrow \GL_2(\mbz_m), \\
\pi_{(m)} : &\GL_2(\hat{\mbz}) \longrightarrow \GL_2(\mbz_{(m)}).
\end{split}
\end{equation}
In most cases we will denote any projection map simply by $\pi$, but on some occasions we will decorate it with subscripts, such as in \eqref{canonicalprojections} or 
\[
\begin{split}
&\pi_{m^\infty, m} : \GL_2(\mbz_m) \longrightarrow \GL_2(\mbz/m\mbz), \\
&\pi_{nm,n} : \GL_2(\mbz/nm\mbz) \longrightarrow \GL_2(\mbz/n\mbz).
\end{split}
\]
The ring $\hat{\mbz}$ is a topological ring under the profinite topology, and the group $\GL_2(\hat{\mbz})$ inherits the structure of a profininte group.  We recall that any open subgroup $G \subseteq \GL_2(\hat{\mbz})$ is a closed subgroup but not conversely.  In general, given any closed subgroup $G \subseteq \GL_2(\hat{\mbz})$, we denote by $G_m \subseteq \GL_2(\mbz_m)$ (resp. by $G_{(m)} \subseteq \GL_2(\mbz_{(m)})$) its image under $\pi_m$ (resp. under $\pi_{(m)}$) as in \eqref{canonicalprojections}.  We denote by $G(m)$ the image of $G$ under the canonical projection
\[
\GL_2(\hat{\mbz}) \longrightarrow \GL_2(\mbz/m\mbz).
\]
For any $m \in \mbn$ and any $d$ dividing $m$, we denote the prime-to-$d$ part of $m$  by
\[
m_{(d)} := \frac{m}{\prod_{\ell \mid d} \ell^{\ord_\ell(m)}}.
\]
Finally, we let 
\[
\begin{split}
\id_m &: \GL_2(\mbz_m) \longrightarrow \GL_2(\mbz_m), \\
\id_{(m)} &: \GL_2(\mbz_{(m)}) \longrightarrow \GL_2 (\mbz_{(m)})
\end{split}
\]
denote the identity maps, and we let $1_{m}$ (resp. $1_{(m)}$) denote the identity element of $\GL_2(\mbz_m)$ (resp. of $\GL_2(\mbz_{(m)})$).  We may also at times denote by $1_m$ the identity element of $\GL_2(\mbz/m\mbz)$.

For an abelian group $A$ and a positive integer $n$ we as usual denote by $A[n]$ the $n$-torsion subgroup of $A$.  For a prime number $\ell$ we define
\[
A[\ell^\infty] := \bigcup_{n = 0}^\infty A[\ell^n], \quad\quad A_{\tors} := \bigcup_{n = 1}^\infty A[n],  \quad\quad A_{\tors,(\ell)} := \bigcup_{{\begin{substack} { n = 1 \\ \ell \nmid n} \end{substack}}}^\infty A[n].
\]
Note that, if $A[n]$ is finite for each $n \in \mbn$, we have 
\[
A_{\tors} \simeq A[\ell^\infty] \times A_{\tors, (\ell)}.
\]

For a number field $K$, we denote by $\mc{O}_K$ its ring of integers, by $\gD_K$ its absolute discriminant and by 
\[
N_{K/\mbq} : K \longrightarrow \mbq
\]
the norm map.
A critical issue that arises in the proof of Proposition \ref{radmGdoubleprimeboundprop} is that of \emph{entanglement} of division fields, i.e. the possibility that the field extension $K \subseteq K(E[m_1]) \cap K(E[m_2])$ is a non-trivial extension, where $m_1$ and $m_2$ are relatively prime positive integers.  Putting $F := K(E[m_1]) \cap K(E[m_2])$, we have by Galois theory that
\[
\gal\left(K(E[m_1m_2]) / K\right) \simeq \left\{ (\gs_1, \gs_2 ) \in \gal\left(K(E[m_1])/K\right) \times  \gal\left(K(E[m_2])/K\right) : \gs_1 |_{F} = \gs_2 |_{F} \right\}.
\]
More generally, if $G_1$, $G_2$ and $H$ are groups and $\psi_1 : G_1 \longrightarrow H$, $\psi_2 : G_2 \longrightarrow H$ are surjective group homomorphisms, we introduce the following notation for the fibered product:
\[
G_1 \times_{\psi} G_2 := \{ (g_1, g_2) \in G_1 \times G_2 : \psi_1(g_1) = \psi_2(g_2) \}
\]
(here $\psi$ is an abbreviation for the ordered pair $(\psi_1,\psi_2)$).  Evidently, $K \neq K(E[m_1]) \cap K(E[m_2])$ if and only if the fibered product
\[
\gal\left(K(E[m_1])/K\right) \times_{\res} \gal\left(K(E[m_2])/K\right)
\]
is a fibered product over a non-trivial group, where 
\[
\res_i : \gal\left(K(E[m_i])/K\right) \longrightarrow \gal\left(K(E[m_1]) \cap K(E[m_2])/K\right)
\]
denotes the restriction map.

\section{Proof of Proposition \ref{mGprimeboundprop}}

In this section we prove Proposition \ref{mGprimeboundprop}, bounding $m_G/\rad'(m_G)$ in terms of the index of $G(m_G)$ in $\pi^{-1}\left(G(\rad'(m_G))\right)$, where $G \subseteq \GL_2(\hat{\mbz})$ is any open subgroup.  We recall that, in the profinite topology, any open subgroup of $\GL_2(\hat{\mbz})$ is necessarily closed; we will establish some lemmas regarding closed subgroups of $\GL_2(\hat{\mbz})$ which thus apply to the open subgroup $G$.

We begin by giving a more precise description of the local exponents $\gb_\ell \geq 0$ occurring in
\begin{equation} \label{defofgbsubell}
m_G =: \prod_{\ell} \ell^{\gb_\ell}.
\end{equation}
In what follows we use the maps
\begin{equation*} 
\begin{split}
&\pi_{\ell^{\gb+1},\ell^{\gb}} \times \id_{(\ell)} : \GL_2(\mbz/\ell^{\gb+1}\mbz) \times \GL_2(\mbz_{(\ell)}) \longrightarrow \GL_2(\mbz/\ell^{\gb}\mbz) \times\GL_2(\mbz_{(\ell)}) \\
&\pi_{\ell^\infty,\ell^{\gb+1}} \times \id_{(\ell)} : \GL_2(\mbz_\ell) \times \GL_2(\mbz_{(\ell)}) \longrightarrow \GL_2(\mbz/\ell^{\gb+1}\mbz) \times \GL_2(\mbz_{(\ell)})
\end{split}
\end{equation*}
defined by the obvious projection in the first factor and the identity map in the second factor.
For any prime $\ell$, we define
\begin{equation} \label{defofalphasubell}
\ga_\ell := 
\begin{cases}
2 & \text{ if } \ell = 2 \\
1 & \text{ if } \ell \geq 3.
\end{cases}
\end{equation}
The next lemma follows from ideas in \cite[Lemma 3, IV-23]{serre2}.  In its statement and henceforth, we will interpret $\GL_2(\mbz/\ell^0\mbz) := \{ 1 \}$ as the trivial group, so that $\ker \pi_{\ell,1} = \GL_2(\mbz/\ell\mbz)$.
\begin{lemma} \label{alphasubelllemma}
Let $G \subseteq \GL_2(\hat{\mbz})$ be a closed subgroup, let $\ell$ be a prime number, and let $\gb \in \mbz_{\geq 0}$. Assume that
\[
\forall \gamma \in [\gb, \max \{ \gb, \ga_\ell \} ] \cap \mbz, \quad \ker( \pi_{\ell^{\gamma+1}, \ell^{\gamma}} ) \times \{ 1_{(\ell)} \} \subseteq (\pi_{\ell^\infty, \ell^{\gamma + 1}} \times \id_{(\ell)}) (G),
\]
where $\ga_\ell$ is as in \eqref{defofalphasubell}.  We then have
\begin{equation*} 
\ker( \pi_{\ell^{\infty}, \ell^{\gb}} ) \times \{ 1_{(\ell)} \} \subseteq G.
\end{equation*}
\end{lemma}
\begin{proof}
Since $G \subseteq \GL_2(\hat{\mbz})$ is closed, it suffices to prove that, for each $n \in \mbz_{\geq \max\{\gb,\ga_\ell\}}$, one has
\begin{equation} \label{kernelinductionhypothesis}
\ker (\pi_{\ell^{n+1},\ell^{n}}) \times \{ 1_{(\ell)} \} \subseteq (\pi_{\ell^\infty, \ell^{n+1}} \times \id_{(\ell)})(G).
\end{equation}
We prove this by induction on $n$ as follows (the base case $n = \max\{\gb,\ga_\ell\}$ is true by hypothesis).  First note that, for $n \geq 1$, we have
\begin{equation} \label{shapeofkernel}
\ker (\pi_{\ell^{n+1},\ell^{n}}) = \{  I + \ell^{n} \tilde{X}  \mod \ell^{n+1} : \tilde{X} \in M_{2\times 2}(\mbz_\ell) \}.
\end{equation}
Thus, \eqref{kernelinductionhypothesis} may be reformulated as saying
\begin{equation} \label{desiredconditioninn}
\forall X \in M_{2\times 2}(\mbf_\ell), \, \exists \tilde{X} \in M_{2\times 2}(\mbz_\ell) \text{ such that } \tilde{X} \equiv X \mod \ell \text{ and } g := (I + \ell^{n}\tilde{X}, 1_{(\ell)}) \in G.
\end{equation}
Our goal is to deduce that \eqref{desiredconditioninn} continues to hold when $n$ is replaced by $n+1$.  Since $G$ is a group, $g^\ell \in G$, and one sees by considering the binomial expansion 
\begin{equation} \label{binomialexpansion}
(I + \ell^{n}\tilde{X})^\ell =  I + \binom{\ell}{1} \ell^n \tilde{X} + \binom{\ell}{2} \ell^{2n} \tilde{X}^2 + \dots + \binom{\ell}{\ell-1} \ell^{(\ell-1)n} \tilde{X}^{\ell-1} + \ell^{\ell n} \tilde{X}^\ell
\end{equation}
that
\[
(\pi_{\ell^\infty,\ell^{n+2}} \times \id_{(\ell)})(g^\ell) = (I + \ell^{n+1}\tilde{X} \pmod{\ell^{n+2}}, 1_{(\ell)}).
\]
Since $X$ in \eqref{desiredconditioninn} was arbitrary, it follows by \eqref{shapeofkernel} that 
\begin{equation*}
\ker (\pi_{\ell^{n+2},\ell^{n+1}}) \times \{ 1_{(\ell)} \} \subseteq (\pi_{\ell^\infty, \ell^{n+2}} \times \id_{(\ell)})(G),
\end{equation*}
completing the induction and proving the lemma.  
\end{proof}
\begin{remark} \label{elladicversionremark}
The ``purely $\ell$-adic version'' of Lemma \ref{alphasubelllemma} also follows by the same proof (without the $\GL_2(\mbz_{(\ell)})$ factor).  Precisely, for any prime $\ell$ and closed subgroup $G \subseteq \GL_2(\mbz_\ell)$, and any $\gb \in \mbz_{\geq 0}$, one has
\begin{equation} \label{GL2elladicversion}
\begin{matrix}
\forall \gamma \in [\gb, \max \{ \gb,\ga_\ell \}] \cap \mbz, \\
\ker\left( \GL_2(\mbz/\ell^{\gamma+1}\mbz) \rightarrow \GL_2(\mbz/\ell^{\gamma}\mbz) \right) \subseteq G(\ell^{\gamma+1}) 
\end{matrix}
\; \Longrightarrow \; \ker\left( \GL_2(\mbz_\ell) \rightarrow \GL_2(\mbz/\ell^{\gb}\mbz) \right) \subseteq G,
\end{equation}
where $\ga_\ell$ is as in \eqref{defofalphasubell}. 
\end{remark}   
\begin{remark}
The fact that the exponent $\ga_\ell$ in \eqref{GL2elladicversion} is different for $\ell = 2$ and otherwise uniform for $\ell \geq 3$ stands in contrast with \cite[Lemma 3, IV-23]{serre2}, which breaks into cases according to whether $\ell \leq 3$ or $\ell \geq 5$.  The underlying reason is that we are seeking to conclude that $\ker\left( \GL_2(\mbz_\ell) \rightarrow \GL_2(\mbz/\ell^{\gb}\mbz) \right) \subseteq G$, rather than the weaker conclusion $\ker\left( \SL_2(\mbz_\ell) \rightarrow \SL_2(\mbz/\ell^{\gb}\mbz) \right) \subseteq G$, the latter breaking into cases according to the condition $
\left[ \SL_2(\mbz/\ell^{\gb}\mbz), \SL_2(\mbz/\ell^{\gb}\mbz) \right] = \SL_2(\mbz/\ell^{\gb}\mbz),
$
which happens if and only if $\ell \geq 5$;
see Lemma \ref{usefulSL2elladicversion} below.
\end{remark}
\begin{Definition} \label{defofgbsubellprime}
We define the
exponents $\gb_\ell' = \gb_\ell'(G)$ by
\begin{equation*} 
\gb_\ell' := \min \{ \gb \in \mbz_{\geq 0} : \, \forall \gamma \in [\gb, \max \{ \gb,\ga_\ell \}] \cap \mbz, \, \ker( \pi_{\ell^{\gamma+1}, \ell^{\gamma}} ) \times \{ 1_{(\ell)} \} \subseteq (\pi_{\ell^\infty, \ell^{\gamma + 1}} \times \id_{(\ell)}) (G) \},
\end{equation*}
where $\ga_\ell$ is as in \eqref{defofalphasubell}.
\end{Definition}
\begin{cor} \label{equalityofgbsubellandgbsubellprimecorollary}
We have $\gb_\ell = \gb_\ell'$, where $\gb_\ell$ is as in \eqref{defofgbsubell}.
\end{cor}
\begin{proof}
By Lemma \ref{alphasubelllemma}, for each prime $\ell$ we have
\[
\ker( \pi_{\ell^{\infty}, \ell^{\gb_\ell'}} ) \times \{ 1_{(\ell)} \} \subseteq G.
\]
Since $\ker \left( \GL_2(\hat{\mbz}) \rightarrow \GL_2(\mbz/ \prod_\ell \ell^{\gb_\ell'}\mbz) \right)$ is equal to the subgroup of $\GL_2(\hat{\mbz})$ generated by $\ker( \pi_{\ell^{\infty}, \ell^{\gb_\ell'}} ) \times \{ 1_{(\ell)} \}$ as $\ell$ varies over all primes, we then have
\[
\ker \left( \GL_2(\hat{\mbz}) \rightarrow \GL_2(\mbz/ \prod_\ell \ell^{\gb_\ell'}\mbz) \right) \subseteq G.
\]
Thus, by \eqref{defofgbsubell} and Definition \ref{definitionofme}, we see that $\gb_\ell \leq \gb_\ell'$.

Conversely, suppose for the sake of contradiction that $\gb_\ell < \gb_\ell'$.  By definition of $\gb_\ell$, we would then have
\begin{equation} \label{definitingconditionforgbsubell}
\ker \left( \pi_{\ell^\infty,\ell^{\gb_\ell'-1}} \right) \times \{ 1_{(\ell)} \}  \subseteq \ker \left( \pi_{\ell^\infty,\ell^{\gb_\ell}} \right) \times \{ 1_{(\ell)} \} \subseteq G.
\end{equation}
Furthermore, since $\pi_{\ell^\infty,\ell^{\gb_\ell'}}\left( \ker ( \pi_{\ell^\infty,\ell^{\gb_\ell'-1}} ) \right) = \ker( \pi_{\ell^{\gb_\ell'},\ell^{\gb_\ell'-1}} )$, we then see that \eqref{definitingconditionforgbsubell} would then imply
\[
\forall \gamma \in [\gb_\ell'-1, \max \{ \gb_\ell'-1,\ga_\ell \}] \cap \mbz, \, \ker( \pi_{\ell^{\gamma+1}, \ell^{\gamma}} ) \times \{ 1_{(\ell)} \} \subseteq (\pi_{\ell^\infty, \ell^{\gamma + 1}} \times \id_{(\ell)}) (G),
\]
contradicting Definition \ref{defofgbsubellprime}.  Thus, $\gb_\ell' \leq \gb_\ell$.
\end{proof}
We will also find it useful to have sufficient conditions to conclude that $\SL_2(\mbz_\ell) \subseteq G$ where $G \subseteq \GL_2(\mbz_\ell)$ is an arbitrary closed subgroup.  The next lemma does so for $\ell$ odd, and gives us sufficient information to allow us to deal separately with the prime $\ell = 2$.  As with Lemma \ref{alphasubelllemma}, it can be largely deduced from arguments found in the proof of \cite[Lemma 3, IV-23]{serre2}; we include the details here for the sake of completeness.
\begin{lemma} \label{usefulSL2elladicversion}
Let $\ell$ be a prime number and let $G \subseteq \GL_2(\mbz_\ell)$ be a closed subgroup.  If $\ell \geq 5$, then we have
\begin{equation*} 
\SL_2(\mbz/\ell\mbz) \subseteq G(\ell) \; \Longrightarrow \; \SL_2(\mbz_\ell) \subseteq G.
\end{equation*}
If $\ell = 3$, we have
\[
G(3) = \GL_2(\mbz/3\mbz) \; \text{ and } \; \SL_2(\mbz/9\mbz) \subseteq G(9) \; \Longrightarrow \; \SL_2(\mbz_3) \subseteq G.
\]
Finally, if $\ell = 2$, we have
\[
G(4) = \GL_2(\mbz/4\mbz) \; \Longrightarrow \; G = \GL_2(\mbz_2) \; \text{ or } \; [\GL_2(\mbz/8\mbz) : G(8)] = 2.
\]
\end{lemma}
\begin{proof}
We first first assume $\ell$ is odd.  Under the stated hypotheses, we will show that $\SL_2(\mbz_\ell) \subseteq G$ by establishing that
\begin{equation} \label{SL2insidecommutators}
\SL_2(\mbz_\ell) = [G,G],
\end{equation}
which amounts to showing that $\SL_2(\mbz_\ell) \subseteq [G,G]$,
since the reverse inclusion follows from the fact that every commutator has determinant $1$.  We begin by first showing, by induction on $n$, that 
\begin{equation} \label{SL2kernelcontainedin}
\ker\left( \SL_2(\mbz/\ell^{n+1}\mbz) \rightarrow \SL_2(\mbz/\ell^{n}\mbz) \right) \subseteq G(\ell^{n+1}) \quad 
\left( \begin{matrix} \ell \geq 5 \text{ and } n \geq 0, \text{ or} \\ \ell = 3 \text{ and } n \geq 1 \end{matrix} \right).
\end{equation}
The binomial expansion argument \eqref{binomialexpansion} of Lemma \ref{alphasubelllemma} shows this, except for the case $\ell \geq 5$ and $n = 0$.  To establish this final case, we first observe that
\[
\det \left( I + \ell^n \tilde{X} \right) \equiv 1 + \ell^n \tr \tilde{X} \mod \ell^{n+1} \quad\quad (n \geq 1).
\]
Thus, for $n \geq 1$, we have
\begin{equation} \label{shapeofSL2kernel}
\ker \left( \SL_2(\mbz/\ell^{n+1}\mbz) \rightarrow \SL_2(\mbz/\ell^n\mbz) \right) = \{  I + \ell^{n} \tilde{X}  \mod \ell^{n+1} : \tilde{X} \in M_{2\times 2}^{\tr \equiv 0}(\mbz_\ell) \},
\end{equation}
where $M_{2\times 2}^{\tr \equiv 0}(\mbz_\ell) := \{ \tilde{X} \in M_{2\times 2}(\mbz_\ell) : \tr \tilde{X} \equiv 0 \mod \ell \}$.  In particular, $\ker \left( \SL_2(\mbz/\ell^{n+1}\mbz) \rightarrow \SL_2(\mbz/\ell^n\mbz) \right)$ is a 3-dimensional subspace of the 4-dimensional $\mbz/\ell\mbz$-vector space $\ker (\pi_{\ell^{n+1},\ell^n})$.  It follows from this, together with the fact that $\begin{pmatrix} 1 & 1 \\ 0 & 1 \end{pmatrix}$ and $\begin{pmatrix} 1 & 0 \\ 1 & 1 \end{pmatrix}$ reduced modulo $\ell$ generate $\SL_2(\mbz/\ell\mbz)$, that the set 
\[
\mc{K} := \left\{ \begin{pmatrix} 0 & 1 \\ 0 & 0 \end{pmatrix}, \begin{pmatrix} 0 & 0 \\ 1 & 0 \end{pmatrix}, \begin{pmatrix} 1 & 1 \\ -1 & -1 \end{pmatrix} \right\} \subseteq M_{2\times 2}(\mbz)
\]
satisfies
\begin{equation} \label{mcKgenerateskernel}
\left\langle I + \ell^n \mc{K} \mod \ell^{n+1} \right\rangle = \ker\left( \SL_2(\mbz/\ell^{n+1}\mbz) \rightarrow \SL_2(\mbz/\ell^n\mbz) \right) \quad\quad
\left( n \geq 0 \right).
\end{equation}
Fix $X \in \mc{K}$.  Note that $I + \ell^0X \mod \ell \in \SL_2(\mbz/\ell\mbz)$, which by hypothesis is contained in $G(\ell)$.  Fix a lift $\tilde{X} \in M_{2\times 2}(\mbz_\ell)$ for which $I + \ell^0\tilde{X} \in G$, and note that $\tilde{X}^2 \equiv {\boldsymbol{0}} \mod \ell$, so $\tilde{X}^4 \equiv {\boldsymbol{0}} \mod \ell^2$.  Thus, since $\ell \geq 5$, we have
\[
(I + \ell^{0}\tilde{X})^\ell = I + \binom{\ell}{1} \tilde{X} + \binom{\ell}{2} \tilde{X}^2 + \dots + \binom{\ell}{\ell-1} \tilde{X}^{\ell-1} + \tilde{X}^\ell \equiv I + \ell \tilde{X} \mod \ell^{2}, 
\]
and more generally,
\[
(I + \ell^{n}\tilde{X})^\ell \equiv I + \ell^{n+1} \tilde{X} \mod \ell^{n+2} \quad
\left( \begin{matrix} \ell \geq 5 \text{ and } n \geq 0, \text{ or} \\ \ell = 3 \text{ and } n \geq 1 \end{matrix} \right).
\]
Therefore \eqref{SL2kernelcontainedin} is established by induction on $n$.

We now proceed to verify \eqref{SL2insidecommutators} for $\ell$ an odd prime.  When $\ell \geq 5$, the group $\PSL_2(\mbz/\ell\mbz)$ is a non-abelian simple group (see e.g. \cite[Ch. II, Hauptsatz 6.13]{huppert}), and the exact sequence
\[
1 \longrightarrow \{ \pm I \} \longrightarrow \SL_2(\mbz/\ell\mbz) \longrightarrow \PSL_2(\mbz/\ell\mbz) \longrightarrow 1
\]
does not split (see e.g. \cite[Lemma 2.3]{zywina}).  From this and a computer computation for the prime $\ell = 3$, we then find that
\begin{equation*}
\begin{split}
\ell \geq 5 \; &\Longrightarrow \; [\SL_2(\mbz/\ell\mbz), \SL_2(\mbz/\ell\mbz)] = \SL_2(\mbz/\ell\mbz), \\
\ell = 3 \; &\Longrightarrow \; [\GL_2(\mbz/3\mbz),\GL_2(\mbz/3\mbz)] = \SL_2(\mbz/3\mbz),
\end{split}
\end{equation*}
and so by the hypotheses of our lemma in this case, we have $[G(\ell),G(\ell)] = \SL_2(\mbz/\ell\mbz)$.
Note further that the commutator subgroup $[G,G] \subseteq G$ projects modulo $\ell$ onto the commutator subgroup $[G(\ell),G(\ell)] $.  We will prove by induction on $n \in \mbn$ that 
\begin{equation} \label{commutatorcontainmentwewant}
[G(\ell^n), G(\ell^n)] = \SL_2(\mbz/\ell^n\mbz) \quad\quad (n \geq 1),
\end{equation}
having just established the base case.  Fix $n \geq 1$ and assume that \eqref{commutatorcontainmentwewant} holds.  Pick any $g \in G(\ell^{n+1})$ and $\tilde{X} \in M^{\tr \equiv 0}_{2\times 2}(\mbz_\ell)$, so that, by \eqref{SL2kernelcontainedin} and \eqref{shapeofSL2kernel}, we have $I + \ell^n\tilde{X} \mod \ell^{n+1} \in G(\ell^{n+1})$.  We then compute the commutator
\begin{equation} \label{shapeofcommutator}
\begin{split}
g (I + \ell^n \tilde{X}) g^{-1} (I + \ell^n \tilde{X})^{-1} 
&\equiv g (I + \ell^n \tilde{X}) g^{-1} (I - \ell^n \tilde{X}) \\
&\equiv I + \ell^n ( g \tilde{X} g^{-1} - \tilde{X} ) \mod \ell^{n+1}.
\end{split}
\end{equation}
Consider the following computations in $M_{2\times 2}(\mbz/\ell\mbz)$:
\[
\begin{split}
\begin{pmatrix} 1 & 0 \\ 1 & 1 \end{pmatrix} \begin{pmatrix} 0 & 1 \\ 0 & 0 \end{pmatrix} \begin{pmatrix} 1 & 0 \\ 1 & 1 \end{pmatrix}^{-1} &= \begin{pmatrix} -1 & 1 \\ -1 & 1 \end{pmatrix}, \\
\begin{pmatrix} 1 & 1 \\ 0 & 1 \end{pmatrix} \begin{pmatrix} 0 & 0 \\ 1 & 0 \end{pmatrix} \begin{pmatrix} 1 & 1 \\ 0 & 1 \end{pmatrix}^{-1} &= \begin{pmatrix} 1 & -1 \\ 1 & -1 \end{pmatrix}, \\
\begin{pmatrix} 0 & -1 \\ 1 & 0 \end{pmatrix} \begin{pmatrix} 1 & 0 \\ 0 & -1 \end{pmatrix} \begin{pmatrix} 0 & -1 \\ 1 & 0 \end{pmatrix}^{-1} &= \begin{pmatrix} -1 & 0 \\ 0 & 1 \end{pmatrix}. 
\end{split}
\]
It follows that, inside the additive 3-dimensional $\mbz/\ell\mbz$-vector space
\[
M_{2\times 2}^{\tr = 0} (\mbz/\ell\mbz) := \{ X \in M_{2\times 2}(\mbz/\ell\mbz) : \tr X = 0 \},
\]
we have
\begin{equation*} 
\ell \geq 3 \; \Longrightarrow \; \left\langle \{ g X g^{-1} - X : g \in \SL_2(\mbz/\ell\mbz), X \in M_{2\times 2}^{\tr = 0} (\mbz/\ell\mbz) \} \right\rangle = M_{2\times 2}^{\tr = 0} (\mbz/\ell\mbz).
\end{equation*}
Thus, varying $g$ and $\tilde{X}$ in \eqref{shapeofcommutator}, we see that
\[
\ker \left( \SL_2(\mbz/\ell^{n+1}\mbz) \rightarrow \SL_2(\mbz/\ell^n\mbz) \right) \subseteq [G(\ell^{n+1}),G(\ell^{n+1}) ],
\]
verifying that \eqref{commutatorcontainmentwewant} holds with $n$ replaced by $n+1$, thus completing the induction step.  Since $[G,G] \subseteq \SL_2(\mbz_\ell)$ is a closed subgroup, we have therefore verified \eqref{SL2insidecommutators}, proving Lemma \ref{usefulSL2elladicversion} in case $\ell$ is odd.

Now assume $\ell = 2$ and note that \eqref{mcKgenerateskernel} is still valid.  By the hypothesis that $G(4) = \GL_2(\mbz/4\mbz)$, for each $X \in \mc{K}$ we may find a lift $\tilde{X} \in M_{2\times 2}(\mbz_2)$ for which $\tilde{X} \equiv X \mod 2$ and $I + 2\tilde{X} \in G$.  Again computing
\[
(I + 2\tilde{X})^2 = I + 4\tilde{X} + 4\tilde{X}^2 \equiv I + 4\tilde{X} \mod 8,
\]
we see that $\ker\left( \SL_2(\mbz/8\mbz) \rightarrow \SL_2(\mbz/4\mbz) \right) \subseteq G(8)$, and it follows that $[\GL_2(\mbz/8\mbz) : G(8) ] \leq 2$.  Finally, if $G(8) = \GL_2(\mbz/8\mbz)$, then \eqref{GL2elladicversion} with $\gb = 0$ implies that $G = \GL_2(\mbz_2)$.
\end{proof}

Next we will employ the following group theoretical lemma.
\begin{lemma} \label{GiHiNilemma}
Let $G_1$ and $G_2$ be finite groups and let $\pi : G_1 \rightarrow G_2$ be a surjective group homomorphism.  Let $H_1 \subseteq G_1$ and $H_2 \subseteq G_2$ be subgroups satisfying $\pi(H_1) = H_2$ and let $N_1 \unlhd G_1$ and $N_2 \unlhd G_2$ be normal subgroups satisfying $\pi(N_1) = N_2$.  Assume that
\begin{equation} \label{keyassumptionsonN1andkerpi}
\gcd(\#N_1, \# \ker \pi) = 1 \quad \text{ and } \quad \left[ N_1, \ker \pi \right] = \{ 1 \}.
\end{equation}
We then have
\[
N_1 \subseteq H_1 \; \Longleftrightarrow \; N_2 \subseteq H_2.
\]
\end{lemma}
\begin{proof}
The implication $\Longrightarrow$ is immediate and does not require \eqref{keyassumptionsonN1andkerpi}.  For the converse, suppose that $N_2 \subseteq H_2$ and let $n_1 \in N_1$.  Since $\pi(N_1) = N_2 \subseteq H_2 = \pi(H_1)$, we see that there exists $h_1 \in H_1$ satisfying $\pi(n_1) = \pi(h_1)$, and we may thus find $k \in \ker \pi$ so that $n_1k \in H_1$.  Now by \eqref{keyassumptionsonN1andkerpi}, we see that
\[
(n_1k)^{\# \ker \pi} = n_1^{\# \ker \pi} \in H_1,
\]
which again by \eqref{keyassumptionsonN1andkerpi} implies that $n_1 \in H_1$.  Thus, $N_1 \subseteq H_1$, proving the lemma.
\end{proof}
Applying Lemma \ref{GiHiNilemma} in a special case, we obtain
\begin{lemma} \label{primedividesindexlemma}
Let $G \subseteq \GL_2(\hat{\mbz})$ be an open subgroup, let $m_G$ be as in Definition \ref{definitionofme} and let $\rad'(m_G)$ be defined by \eqref{defofradprime}.
For any prime $\ell$ and $d \in \mbn$, one has
\[
\rad'(m_G) \mid d \mid d\ell \mid m_G \; \Longrightarrow \; \ell \text{ divides } \left[ \pi_{\ell d, d}^{-1}(G(d)) : G(\ell d) \right].
\]
\end{lemma}
\begin{proof}
We write $m  := m_G$ and 
\[
d =: \ell^{\gd_\ell} \cdot d_{(\ell)}, \quad m =: \ell^{\gb_\ell} \cdot m_{(\ell)}
\]
(where $\ell \nmid d_{(\ell)} m_{(\ell)}$), and note that, by hypothesis, $\ga_\ell \leq \gd_\ell < \gb_\ell$.  Further observe that, since $\gb_\ell = \gb_\ell'$, by Definition \ref{definitionofme} and Definition \ref{defofgbsubellprime}, we have
\[
\ker \left( \pi_{\ell^{\gd_\ell+1},\ell^{\gd_\ell}} \right) \times \{ 1_{m_{(\ell)}} \} \not\subseteq G(\ell^{\gd_\ell+1} m_{(\ell)}).
\]
We now apply Lemma \ref{GiHiNilemma} with
\[
\begin{array}{lll}
G_1 := \GL_2(\mbz/\ell^{\gd_\ell+1}m_{(\ell)}\mbz), &H_1 := G(\ell^{\gd_\ell+1}m_{(\ell)}), &N_1 := \ker( \pi_{\ell^{\gd_\ell+1},\ell^{\gd_\ell}} ) \times \{ 1_{m_{(\ell)}} \}, \\
G_2 := \GL_2(\mbz/\ell^{\gd_\ell+1}d_{(\ell)}\mbz), 
&H_2 := G(\ell^{\gd_\ell+1}d_{(\ell)}), &N_2 := \ker( \pi_{\ell^{\gd_\ell+1},\ell^{\gd_\ell}} ) \times \{ 1_{d_{(\ell)}} \},
\end{array}
\]
and $\pi : \GL_2(\mbz/\ell^{\gd_\ell+1}m_{(\ell)}\mbz) \longrightarrow \GL_2(\mbz/\ell^{\gd_\ell+1}d_{(\ell)}\mbz)$ the canonical projection map.  The conclusion is that
\[
\ker ( \pi_{\ell^{\gd_\ell+1},\ell^{\gd_\ell}} ) \times \{ 1_{d_{(\ell)}} \} \not\subseteq G(\ell^{\gd_\ell+1} d_{(\ell)}).
\]
Since $\ker ( \pi_{\ell^{\gd_\ell+1},\ell^{\gd_\ell}} ) \times \{ 1_{d_{(\ell)}} \} \simeq \ker ( \pi_{\ell d, d} )$ is an $\ell$-group, this proves the lemma.
\end{proof}
Applying Lemma \ref{primedividesindexlemma} prime by prime, for each prime $\ell$ dividing $m_G/\rad'(m_G)$, we obtain
\[
\frac{m_G}{\rad'(m_G)} \; \text{ divides } \; \left[ \pi^{-1}\left( G(\rad'(m_G)) \right) : G( m_G ) \right],
\]
proving Proposition \ref{mGprimeboundprop} in the case that \eqref{theconditionat3} does not hold.  In case \eqref{theconditionat3} does hold, we have $G(3) = \GL_2(\mbz/3\mbz)$ and, by Lemma \ref{usefulSL2elladicversion}, we must also have $\SL_2(\mbz/9\mbz) \not \subseteq G(9)$.  A computer search reveals that, up to conjugation in $\GL_2(\mbz/9\mbz)$, there are two subgroups $G_1, G_2 \subseteq \GL_2(\mbz/9\mbz)$ meeting these two criteria\footnote{The (genus zero) modular curve associated with $G_2$ has been considered by N. Elkies \cite{elkies}, who exhibited an explicit map from it to the $j$-line.}, and $G_1 \subseteq G_2$.  Furthermore, $[\GL_2(\mbz/9\mbz) : G_2] = 27$.   From this it follows that $27$ divides $[\GL_2(\mbz/9\mbz) : G(9)]$, and so
\[
9 \cdot 3  \; \text{ divides } \; \left[ \pi^{-1}\left(G(\rad'(m_G))\right) : G( 3\rad'(m_G) ) \right].
\] 
Now starting here and applying Lemma \ref{primedividesindexlemma}, prime by prime, we conclude the proof of Proposition \ref{mGprimeboundprop} in the case that \eqref{theconditionat3} holds.

\section{Proof of Proposition \ref{radmGdoubleprimeboundprop}}

We now prove Proposition \ref{radmGdoubleprimeboundprop}.  The proof will rely, in part, on the following corollary to the N\'{e}ron-Ogg-Shafarevich criterion (see for instance \cite{ogg} or \cite[Ch. VII, Theorem 7.1]{silverman}).
\begin{Theorem} \label{neronoggshafarevichthm}
Let $K$ be a number field, let $E$ be an elliptic curve over $K$ and let $\mc{L} \subseteq \mc{O}_K$ be a prime ideal of $K$, lying over the rational prime $\ell$ of $\mbz$.  The following are equivalent:
\begin{enumerate}
\item[(a)] $E$ has good reduction at $\mc{L}$.
\item[(b)] For each positive integer $m$ that is not divisible by $\ell$, the prime $\mc{L}$ is unramified in $K(E[m])$.
\item[(c)] The prime $\mc{L}$ is unramified in $K(E_{\tors,(\ell)})$.
\end{enumerate}
\end{Theorem}

We presently reduce the proof of Proposition \ref{radmGdoubleprimeboundprop} to the following four lemmas.  The first lemma follows immediately from the classification subgroups of $\GL_2(\mbz/\ell\mbz)$.
\begin{lemma} \label{SL2notcontainedinlemma}
Let $\ell$ be a prime number and let $G(\ell) \subseteq \GL_2(\mbz/\ell\mbz)$ be any subgroup.  We have
\[
\SL_2(\mbz/\ell\mbz) \not \subseteq G(\ell) \; \Longrightarrow \; \ell \leq [ \GL_2(\mbz/\ell\mbz) : G(\ell) ].
\]
\end{lemma}
The second lemma is a consequence of the Weil pairing on an elliptic curve.
\begin{lemma} \label{elldividesdiscKlemma}
Let $E$ be an elliptic curve defined over a number field $K$, let $G := \rho_{E,K}(G_K) \subseteq \GL_2(\hat{\mbz})$, and let $\ell$ be a prime number.  For any positive integer $n$, we have
\[
\SL_2(\mbz/\ell^n\mbz) \subseteq G(\ell^n) \neq \GL_2(\mbz/\ell^n\mbz) \; \Longrightarrow \; \ell \mid \gD_K.
\]
Consequently,
\[
\SL_2(\mbz_{\ell}) \subseteq G_\ell \neq \GL_2(\mbz_{\ell}) \; \Longrightarrow \; \ell \mid \gD_K.
\]
\end{lemma}
Our third lemma utilizes the N\'{e}ron-Ogg-Shafarevich criterion in the form of Theorem \ref{neronoggshafarevichthm}.  
\begin{lemma} \label{neronoggshafarevichlemmaellgeq5}
Let $E$ be an elliptic curve defined over a number field $K$, let $G := \rho_{E,K}(G_K) \subseteq \GL_2(\hat{\mbz})$, let $m_G$ be as in Definition \ref{definitionofme} and let $\ell$ be an odd prime number dividing $m_G$.  We then have
\[
G_{\ell} = \GL_2(\mbz_{\ell}) \; \Longrightarrow \; \ell \mid \gD_K N_{K/\mbq}(\gD_E).
\]
\end{lemma}
For the prime $\ell = 2$ we must make a finer analysis, in the form of the next (and final) lemma.  Let us make the abbreviation
\begin{equation*} \label{defofrprime}
r' := \rad'(m_G).
\end{equation*}
\begin{lemma} \label{ellequals2lemma}
Let $E$ be an elliptic curve defined over a number field $K$, let $G := \rho_{E,K}(G_K) \subseteq \GL_2(\hat{\mbz})$, let $m_G$ be as in Definition \ref{definitionofme} and assume that $4$ divides $m_G$.  We then have
\[
\GL_2(\mbz/4\mbz) \times \{ 1_{r'_{(2)}} \} \not \subseteq G(r') \; \Longrightarrow \; 4 \leq 2 \left[ \pi^{-1} \left( G\left(r'_{(2)} \right) \right) : G\left(r' \right) \right]
\]
and
\[
\GL_2(\mbz/4\mbz) \times \{ 1_{r'_{(2)}} \} \subseteq G(r') \; \Longrightarrow \; 2 \mid \gD_K N_{K/\mbq}(\gD_E).
\]
\end{lemma}

Let us now deduce Proposition \ref{radmGdoubleprimeboundprop} from Lemmas \ref{SL2notcontainedinlemma} -- \ref{ellequals2lemma}, postponing the proofs of those lemmas until later.  First, combining Lemma \ref{usefulSL2elladicversion} with Lemmas \ref{SL2notcontainedinlemma}, \ref{elldividesdiscKlemma} and \ref{neronoggshafarevichlemmaellgeq5}, one concludes the following implications, for any prime $\ell \geq 5$ that divides $m_G$:
\begin{equation*} 
\begin{split}
&\SL_2(\mbz/\ell\mbz) \not \subseteq G(\ell) \; \Longrightarrow \; \ell \leq [ \GL_2(\mbz/\ell\mbz) : G(\ell) ], \\
&\SL_2(\mbz/\ell\mbz) \subseteq G(\ell)  \; \Longrightarrow \; \ell \mid \gD_K N_{K/\mbq}(\gD_E).
\end{split}
\end{equation*}
This implies that
\begin{equation} \label{boundforprimesabove5}
\begin{split}
r'_{(6)} &\leq \prod_{{\begin{substack} {\ell \geq 5, \, \ell \mid r' \\ \SL_2(\mbz/\ell\mbz) \not \subseteq G(\ell) } \end{substack}}}  [ \GL_2(\mbz/\ell\mbz) : G(\ell) ] \prod_{{\begin{substack} {\ell \geq 5 \\\ell \mid \gD_K N_{K/\mbq}(\gD_E) \\ \SL_2(\mbz/\ell\mbz) \subseteq G(\ell) } \end{substack}}} \ell \\
&\leq [ \GL_2(\mbz/r'_{(6)}\mbz) : G(r'_{(6)}) ] \rad \left( | \gD_K N_{K/\mbq}(\gD_E) | \right)_{(6)}.
\end{split}
\end{equation}

If the prime $\ell = 3$ divides $m_G$ then either condition \eqref{theconditionat3} holds or it does not hold.  Let us first assume that \eqref{theconditionat3} does not hold, i.e. we assume that it is \emph{not} the case that $9$ divides $m_G$, $G(3) = \GL_2(\mbz/3\mbz)$ and $\SL_2(\mbz_3) \not \subseteq G_3$.  We then use Lemmas \ref{SL2notcontainedinlemma}, \ref{elldividesdiscKlemma} and \ref{neronoggshafarevichlemmaellgeq5}, together with Lemma \ref{usefulSL2elladicversion}, to deduce the following implications:
\begin{equation*} 
\begin{split}
&\SL_2(\mbz/3\mbz) \not \subseteq G(3) \; \Longrightarrow \; 3 \leq [ \GL_2(\mbz/3\mbz) : G(3) ], \\
&\SL_2(\mbz/3\mbz) \subseteq G(3) \neq \GL_2(\mbz/3\mbz) \; \Longrightarrow \; 3 \mid \gD_K, \\
&G(3) = \GL_2(\mbz/3\mbz) \; \text{ and } \; \SL_2(\mbz_3) \subseteq G_3 \neq \GL_2(\mbz_3) \; \Longrightarrow \; 3 \mid \gD_K, \\
&G(3) = \GL_2(\mbz/3\mbz) \; \text{ and } \; G_3 = \GL_2(\mbz_3) \; \Longrightarrow \; 3 \mid N_{K/\mbq}(\gD_E).
\end{split}
\end{equation*}
Inserting this information into \eqref{boundforprimesabove5}, we find that
\begin{equation} \label{boudnforoddprimes}
r'_{(2)} \leq [ \GL_2(\mbz/r'_{(2)}\mbz) : G(r'_{(2)}) ] \rad \left( | \gD_K N_{K/\mbq}(\gD_E) | \right)_{(2)}.
\end{equation}
On the other hand, in case \eqref{theconditionat3} \emph{does} hold, then we obviously have
\begin{equation} \label{incasetheconditionat3doeshold}
\begin{split}
\frac{r'_{(2)}}{3} = r'_{(6)} &\leq [ \GL_2(\mbz/r'_{(6)}\mbz) : G(r'_{(6)}) ] \rad \left( | \gD_K N_{K/\mbq}(\gD_E) | \right)_{(6)} \\
&\leq [ \GL_2(\mbz/r'_{(2)}\mbz) : G(r'_{(2)}) ] \rad \left( | \gD_K N_{K/\mbq}(\gD_E) | \right)_{(2)}.
\end{split}
\end{equation}
If $\ell = 2$ divides $m_G$, then either $4 \mid m_G$ or not.  If $4 \nmid m_G$, then multiplying both sides of \eqref{boudnforoddprimes} (resp. of \eqref{incasetheconditionat3doeshold}) by 2, we obtain the bound of Proposition \ref{radmGdoubleprimeboundprop}.  Now assume that $4 \mid m_G$.  
In this case, when \eqref{theconditionat3} does not hold, we insert the result of Lemma \ref{ellequals2lemma} into \eqref{boudnforoddprimes}, concluding that
\begin{equation*}
r' = 4r'_{(2)} \leq 2 [ \GL_2(\mbz/r'\mbz) : G(r') ] \rad \left( | \gD_K N_{K/\mbq}(\gD_E) | \right).
\end{equation*}
Likewise, in case condition \eqref{theconditionat3} does hold, we insert these results into \eqref{incasetheconditionat3doeshold} and obtain
\begin{equation*}
\frac{r'}{3} = \frac{4r'_{(2)}}{3} \leq 2 [ \GL_2(\mbz/r'\mbz) : G(r') ] \rad \left( | \gD_K N_{K/\mbq}(\gD_E) | \right).
\end{equation*}
Thus we see that Lemmas \ref{SL2notcontainedinlemma} -- \ref{ellequals2lemma} indeed imply Proposition \ref{radmGdoubleprimeboundprop}.  

We now prove each of these lemmas.
First we state an auxiliary lemma that is used throughout and may be found in \cite[Lemma (5.2.1)]{ribet}.
\begin{lemma} \label{goursatlemma} (Goursat's Lemma)
\noindent
Let $G_1$, $G_2$ be groups and for $i \in \{1, 2 \}$ denote by $\pr_i : G_1 \times G_2 \longrightarrow G_i$ the projection map onto the $i$-th factor. Let $G \subseteq G_1 \times G_2$ be a subgroup and 
assume that 
$$
\pr_1(G) = G_1, \; \pr_2(G) = G_2.
$$
Then 
there exists a group $\Gamma$ together with a pair of surjective homomorphisms 
\[
\begin{split}
\psi_1 : G_1 &\longrightarrow \Gamma \\
\psi_2 : G_2 &\longrightarrow \Gamma
\end{split}
\]
so that 
\[
G = G_1 \times_\psi G_2 := \{ (g_1,g_2) \in G_1 \times G_2 : \psi_1(g_1) = \psi_2(g_2) \}.
\]
\end{lemma}

\subsection{Proof of Lemma \ref{SL2notcontainedinlemma}}

To prove Lemma \ref{SL2notcontainedinlemma}, we will use the following classification of certain proper subgroups of $\GL_2$.  
\begin{Definition}
Let $\ell$ be any prime number.
\begin{itemize}
\item[(i)]
A subgroup $G(\ell) \subseteq \GL_2(\mbz/\ell\mbz)$ is called a {\bf{Borel subgroup}} if it is conjugate in 
$\GL_2(\mbz/\ell\mbz)$ to the subgroup 
\begin{equation} \label{defofBofell}
\mc{B}(\ell) := \left\{
\begin{pmatrix}
a & b \\
0 & d
\end{pmatrix} : \;
b \in \mbz/\ell\mbz, \; a, d \in (\mbz/\ell\mbz)^\times
\right\}.
\end{equation}
\\
\item[(ii)]
A subgroup $G(\ell) \subseteq \GL_2(\mbz/\ell\mbz)$ is called a {\bf{normalizer of a split Cartan subgroup}} if it is conjugate in $\GL_2(\mbz/\ell\mbz)$ to the subgroup 
\begin{equation} \label{defofNsubsofell}
\mc{N}_{\s}(\ell) := \left\{
\begin{pmatrix}
a & 0 \\
0 & d
\end{pmatrix} : \; a, d \in (\mbz/\ell\mbz)^\times
\right\}
\cup
\left\{
\begin{pmatrix}
0 & b \\
c & 0
\end{pmatrix} : \; b, c \in (\mbz/\ell\mbz)^\times
\right\}.
\end{equation}
If $\ell$ is odd, then $G(\ell)$ is called a {\bf{normalizer of a non-split Cartan subgroup}} if it is conjugate in $\GL_2(\mbz/\ell\mbz)$ to the subgroup 
\begin{equation} \label{defofNsubnsofell}
\mc{N}_{\ns}(\ell) := \left\{
\begin{pmatrix}
x & \ve y \\
y & x
\end{pmatrix} : \; 
\begin{matrix}
x, y \in \mbz/\ell\mbz, \\ 
x^2-\ve y^2 \neq 0
\end{matrix}
\right\}
\cup
\left\{
\begin{pmatrix}
x & -\ve y \\
y & -x
\end{pmatrix} : \; 
\begin{matrix}
x, y \in \mbz/\ell\mbz, \\ 
x^2-\ve y^2 \neq 0
\end{matrix}
\right\},
\end{equation}
where $\ve$ is any fixed non-square in $(\mbz/\ell\mbz)^\times$.  If $\ell = 2$, then $G(2)$ is called a normalizer of a non-split Cartan subgroup if $G(2) = \GL_2(\mbz/2\mbz)$.
\\
\item[(iii)]
 A subgroup $G(\ell) \subseteq \GL_2(\mbz/\ell\mbz)$ is called an {\bf{exceptional group}} if its image in $\PGL_2(\mbz/\ell\mbz)$ is isomorphic to one of the groups $A_4$, $S_4$ or $A_5$ (the symmetric or alternating groups).  
 \end{itemize}
\end{Definition}
The following lemma may be deduced from Propositions 15, 16 and Section 2.6 of \cite{serre}.  
\begin{lemma} \label{subgroupsofGL2lemma}
Let $G(\ell) \subseteq \GL_2(\mbz/\ell\mbz)$ be a subgroup.  Then one of the following must hold:
\begin{enumerate}
\item $G(\ell)$ is contained in a Borel subgroup.
\item $G(\ell)$ is contained in the normalizer of a split Cartan subgroup.
\item $G(\ell)$ is contained in the normalizer of a non-split Cartan subgroup.
\item $G(\ell)$ is an exceptional group.
\item $\SL_2(\mbz/\ell\mbz) \subseteq G(\ell)$.
\end{enumerate}
\end{lemma}
We include the following table of indices $[\GL_2(\mbz/\ell\mbz) : G(\ell)]$, for each of the proper subgroups $G(\ell)$ given in Lemma \ref{subgroupsofGL2lemma}.  In addition to the definitions \eqref{defofBofell}, \eqref{defofNsubsofell}, and \eqref{defofNsubnsofell}, we make the following abbreviations.  For a prime $\ell$ for which $A_4 \subseteq \PGL_2(\mbz/\ell\mbz)$, we define the exceptional subgroup $\mc{E}_{A_4}(\ell) \subseteq \GL_2(\mbz/\ell\mbz)$ by
\[
\mc{E}_{A_4}(\ell) := \{ g \in \GL_2(\mbz/\ell\mbz) : \varpi(g) \in A_4 \},
\]
where $\varpi : \GL_2(\mbz/\ell\mbz) \longrightarrow \PGL_2(\mbz/\ell\mbz)$ denotes the usual projection.  The exceptional subgroups $\mc{E}_{S_4}(\ell)$ and $\mc{E}_{A_5}(\ell)$ are defined similarly.
\[
\begin{array}{|c||c|c|c|c|c|c|} \hline G(\ell) & \mc{B}(\ell) & \mc{N}_{\s}(\ell) & \mc{N}_{\ns}(\ell) & \mc{E}_{A_4}(\ell) & \mc{E}_{S_4}(\ell) & \mc{E}_{A_5}(\ell) \\ 
\hline \hline [\GL_2(\mbz/\ell\mbz) : G(\ell)] & \ell + 1 & \ell(\ell+1)/2 & \ell(\ell-1)/2 & \ell(\ell^2-1)/12 & \ell(\ell^2-1)/24 & \ell(\ell^2-1)/60 \\ 
\hline
\end{array}
\]
We note that $\mc{N}_{\ns}(2) = \GL_2(\mbz/2\mbz)$, and also that each exceptional group only occurs for certain primes $\ell$.  In particular, if the expression given in the table is not a whole number, then the associated exceptional group does not occur as a subgroup of $\GL_2(\mbz/\ell\mbz)$ for that prime $\ell$.  The conclusion of Lemma \ref{SL2notcontainedinlemma} follows immediately from this table. 

\subsection{Proof of Lemma \ref{elldividesdiscKlemma}}

We will make use of the following commutative diagram, where 
\[
\res : \gal(K(E_{\tors})/K) \rightarrow \gal(K(\mu_\infty)/K), \quad\quad \cyc : \gal(K(\mu_\infty)/K) \rightarrow \hat{\mbz}^\times 
\]
denote respectively the restriction map and the cyclotomic character (the containment $K(\mu_\infty) \subseteq K(E_{\tors})$ follows from the Weil Pairing \cite{weil}, see also \cite[Ch. III, \S 8]{silverman}).
\begin{equation} \label{commutingdiagram}
\begin{CD}
\gal(K(E[\ell^n])/K) @>\rho_{E,K}>> \GL_2(\mbz/\ell^n\mbz) \\
@V{\res}VV @V{\det}VV \\
\gal(K(\mu_{\ell^n})/K) @>\cyc>> (\mbz/\ell^n\mbz)^\times
\end{CD}
\end{equation}

By considering the commutative diagram \eqref{commutingdiagram} and Galois theory, we see that
\[
\SL_2(\mbz/\ell^n\mbz) \subseteq G(\ell^n) \neq \GL_2(\mbz/\ell^n\mbz) \; \Longrightarrow \; \det(G(\ell^n)) \neq (\mbz/\ell^n\mbz)^\times \; \Longrightarrow \; \mbq \neq \mbq(\mu_{\ell^n}) \cap K.
\]
Since $\mbq(\mu_{\ell^n})$ is totally ramified at $\ell$, it follows that $\ell$ is then ramified in $\mbq(\mu_{\ell^n}) \cap K$, so $\ell$ is ramified in $K$, and thus $\ell$ divides $\gD_K$.

\subsection{Proof of Lemma \ref{neronoggshafarevichlemmaellgeq5}}

In order to prove Lemma \ref{neronoggshafarevichlemmaellgeq5}, we will make use of the following definition and lemma, which allow us to understand in more detail the nature of the fibered products that may be present in $G$.
\begin{Definition}
Let $G$ be a profinite group and $\Sigma$ a finite simple group.  We say that \emph{$\Sigma$ occurs in $G$} if and only if there are closed subgroups $G_1$ and $N_1$ of $G$ with $N_1 \subseteq G_1 \subseteq G$, $N_1$ normal in $G_1$ and $G_1 / N_1 \simeq \Sigma$.  We further define
\[
\begin{split}
\occ(G) &:= \{ \text{finite simple non-abelian groups } \Sigma : \; \Sigma \text{ occurs in } G \}, \\
\occ_{\JH}(G) &:= \{ \text{finite simple non-abelian groups } \Sigma : \; \Sigma \text{ is a Jordan-H\"{o}lder factor of $G$} \}.
\end{split}
\]
\end{Definition}
Note that any simple Jordan-H\"{o}lder factor of $G$ occurs in $G$ (but generally not vice versa), i.e. we have 
$
\occ_{\JH}(G) \subseteq \occ(G).
$
Also note that, if 
\[
1 \longrightarrow G' \longrightarrow G \longrightarrow G'' \longrightarrow 1
\]
is an exact sequence of profinite groups, then
\begin{equation} \label{JHobservation}
\begin{split}
\occ(G) &= \occ(G') \cup \occ(G''), \\
\occ_{\JH}(G) &= \occ_{\JH}(G') \cup \occ_{\JH}(G'').
\end{split}
\end{equation}
Finally, as observed in \cite[IV-25]{serre2}, one has that
\begin{equation*} 
\occ(\GL_2(\mbz_{\ell})) = 
\begin{cases}
\emptyset & \text{ if } \ell \in \{ 2, 3 \} \\
\{ \PSL_2(\mbz/5\mbz) \} = \{ A_5 \} & \text{ if } \ell = 5 \\
\{ \PSL_2(\mbz/\ell\mbz) \} & \text{ if } \ell > 5, \ell \equiv \pm 2 \pmod{5} \\
\{ \PSL_2(\mbz/\ell\mbz), A_5 \} & \text{ if } \ell > 5, \ell \equiv \pm 1 \pmod{5}.
\end{cases}
\end{equation*}
Thus, by \eqref{JHobservation} we have
\begin{equation} \label{occincomplementofell}
\occ(\GL_2(\mbz_{(\ell)})) = \{ A_5 \} \cup \{ \PSL_2(\mbz/p\mbz) \}_{p \neq \ell}.
\end{equation}

\begin{lemma} \label{PSL2Zmodplemma}
Let $\ell \geq 5$ be a prime and let $G \subseteq \GL_2(\mbz_{\ell})$ be a closed subgroup satisfying $\SL_2(\mbz/\ell\mbz) \subseteq G(\ell)$.
Suppose further that $\psi : G \longrightarrow H$ is a surjective group homomorphism onto a finite group $H$.  Then either 
\begin{enumerate}
\item $\PSL_2(\mbz/\ell\mbz) \in \occ_{\JH}(H)$, or 
\item $H$ is abelian and $\SL_2(\mbz_{\ell}) \subseteq \ker \psi$.
\end{enumerate}
\end{lemma}
\begin{proof}
As observed earlier, since $\ell \geq 5$, the group $\PSL_2(\mbz/\ell\mbz)$ is a simple non-abelian group, and we obviously have $\{ \PSL_2(\mbz/\ell\mbz) \} \subseteq \occ_{\JH}(\SL_2(\mbz/\ell\mbz))$.  Furthermore, by the hypothesis $\SL(\mbz/\ell\mbz) \subseteq G(\ell)$ together with \eqref{JHobservation}, we see that $\{ \PSL_2(\mbz/\ell\mbz) \} \subseteq \occ_{\JH}(G(\ell))$.  Thus, again by \eqref{JHobservation}, we have 
\begin{equation} \label{PSL2occursinG}
\{ \PSL_2(\mbz/\ell\mbz) \} \subseteq \occ_{\JH}(G).
\end{equation} 
Furthermore, we have that
\begin{equation} \label{eitherorpsl2}
\frac{\pm (\ker \psi)(\ell) \cap \SL_2(\mbz/\ell\mbz)}{\{\pm I\}} 
=
\begin{cases}
\{ 1 \} \; \text{ or } \\
\PSL_2(\mbz/\ell\mbz).
\end{cases}
\end{equation}
If the left-hand side of \eqref{eitherorpsl2} is trivial, then $\ker \psi$ is prosolvable (so that $\occ_{\JH}(\ker\psi) = \emptyset$), and considering 
the exact sequence
\[
1 \longrightarrow \ker \psi \longrightarrow G \longrightarrow H \longrightarrow 1,
\]
we see by \eqref{JHobservation} and \eqref{PSL2occursinG} that $\PSL_2(\mbz/\ell\mbz) \in \occ_{\JH}(H)$.  
If, on the other hand, we have $\PSL_2(\mbz/\ell\mbz)$ in \eqref{eitherorpsl2}, then 
$\SL_2(\mbz/\ell\mbz) \subseteq (\ker \psi)(\ell)$, which by Lemma \ref{usefulSL2elladicversion} applied to $G = \ker \psi$ implies 
that $\SL_2(\mbz_\ell) \subseteq \ker \psi$.  Thus  $H$ is abelian and $\psi$ factors 
through the determinant map, as asserted.
\end{proof}

We now proceed with the proof of Lemma \ref{neronoggshafarevichlemmaellgeq5}.  By Lemma \ref{goursatlemma}, the hypothesis that $G_\ell = \GL_2(\mbz_{\ell})$ and that $\ell$ divides $m_G$ imply that
\begin{equation} \label{fiberedproductatell2}
G \simeq \GL_2(\mbz_{\ell}) \times_{\psi} G_{(\ell)},
\end{equation} 
where $\psi_{\ell} : \GL_2(\mbz_{\ell}) \twoheadrightarrow H$ and $\psi_{(\ell)} : G_{(\ell)} \twoheadrightarrow H$ are surjective homomorphisms onto a common \emph{non-trivial} group $H$.   Under the Galois correspondence, we have $\GL_2(\mbz_{\ell}) = \gal(K(E[\ell^\infty])/K)$, $G_{(\ell)} = \gal(K(E_{\tors,(\ell)})/K)$ and $H = \gal(F/K)$, where $F := K(E[\ell^\infty]) \cap K(E_{\tors,(\ell)}) \neq K$.  Thus, the corresponding field diagram is as follows.
\begin{equation} \label{fancyfielddiagramwithF}
\begin{tikzpicture} 
\matrix(a)[matrix of math nodes, row sep=1.5em, column sep=1.5em, text height=1.5ex, text depth=0.25ex]
{K(E[\ell^\infty]) & & K(E_{\tors,(\ell)}) \\  & F \\ & K\\};
\path[-](a-1-1) edge (a-2-2);
\path[-](a-1-3) edge (a-2-2);
\path[-](a-2-2) edge (a-3-2);
\end{tikzpicture}
\end{equation}
We first claim that 
\begin{equation} \label{FcapKmuinftyneqK}
F \cap  K(\mu_{\ell^\infty}) \neq K.
\end{equation}
We separate the verification of \eqref{FcapKmuinftyneqK} into cases.

\noindent \emph{Case: $\ell \geq 5$.}  By Lemma \ref{PSL2Zmodplemma}, we see that either $\PSL_2(\mbz/\ell\mbz)$ occurs in $H$ (and thus occurs in $G_{(\ell)}$), or $H$ is abelian and $F \subseteq K(\mu_{\ell^\infty})$.  If $\ell \geq 7$ then, by \eqref{occincomplementofell}  we see that $H$ must be abelian and $F \subseteq K(\mu_{\ell^\infty})$, verifying \eqref{FcapKmuinftyneqK}.  If $\ell = 5$, we consider the further quotient induced by reduction modulo 5:
\[
H \simeq \frac{\GL_2(\mbz_5)}{\ker \psi_5} \longrightarrow \frac{\GL_2(\mbz/5\mbz)}{( \ker \psi_5 )(5)} =: H(5).
\]
Since the kernel of this quotient is pro-solvable, we see that if $\PSL_2(\mbz/5\mbz) \simeq A_5$ occurs in $H$, then it must occur in $H(5)$, and a computer calculation shows that we then must have
\[
(\ker \psi_5)(5) \subseteq \left\{ \begin{pmatrix} \gl & 0 \\ 0 & \gl \end{pmatrix} : \gl \in (\mbz/5\mbz)^\times \right\},
\]
and thus
\[
\langle \SL_2(\mbz_5) , \ker \psi_5 \rangle \subseteq \left\{ g \in \GL_2(\mbz_5) : \left( \frac{\det(g) \mod 5}{5} \right) = 1 \right\}.
\]
By the Galois correspondence, we then have
\[
F \cap K(\mu_{5^\infty}) = K(E[5^\infty])^{\langle \SL_2(\mbz_5), \ker \psi_5 \rangle} \supseteq K(\sqrt{5}) \neq K,
\]
where we are using the fact that $G(5) = \GL_2(\mbz/5\mbz)$, which precludes the possibility that $K(\sqrt{5}) = K$.  Thus in any case, \eqref{FcapKmuinftyneqK} also holds for $\ell = 5$.

\noindent \emph{Case: $\ell = 3$.}  As in the previous case, we have that \eqref{fiberedproductatell2} holds with $\ell = 3$.  By Galois theory, we have that
\[
F := K(E[3^{\infty}])^{\ker \psi_3} \supseteq K(E[3^{\infty}])^{\langle \SL_2(\mbz_3), \ker \psi_3 \rangle}  =  K(\mu_{3^\infty}) \cap F.
\]
As in the previous case, since $\ker \psi_3 \neq \GL_2(\mbz_3)$, we have $F \neq K$.  The following lemma will then imply that $K(\mu_{3^\infty}) \cap F \neq K$.

\begin{lemma} \label{kernelat3lemma}
Let $N \unlhd \GL_2(\mbz_3)$ be a closed normal subgroup satisfying $\langle \SL_2(\mbz_3), N \rangle = \GL_2(\mbz_3)$.  Then $N = \GL_2(\mbz_3)$.
\end{lemma}
\begin{proof}
A computer calculation shows that, if $H \unlhd \GL_2(\mbz/9\mbz)$ is a normal subgroup satisfying 
\[
\langle \SL_2(\mbz/9\mbz), H \rangle = \GL_2(\mbz/9\mbz),
\]
then $H = \GL_2(\mbz/9\mbz)$.  Taking $N$ as in the statement of the lemma and setting $H := N(9)$, we see that $N(9) = \GL_2(\mbz/9\mbz)$, and applying \eqref{GL2elladicversion} with $\gb = 1$, we conclude that $N = \GL_2(\mbz_3)$.
\end{proof}
Applying Lemma \ref{kernelat3lemma} with $N = \ker \psi_3$, we find that $K(\mu_{3^\infty}) \cap F \neq K$, since $F \neq K$, verifying \eqref{FcapKmuinftyneqK} in the $\ell = 3$ case as well.  

Finally, we observe that \eqref{FcapKmuinftyneqK} implies the conclusion of Lemma \ref{neronoggshafarevichlemmaellgeq5}.  Indeed, let $\ell$ be any odd prime with $\ell \nmid \gD_K$.  Since $G_\ell = \GL_2(\mbz_\ell)$, we have $K \cap \mbq(\mu_{\ell^\infty}) = \mbq$, and so any prime $\mf{L} \subseteq O_K$ over $\ell$ is totally ramified in $K(\mu_{\ell^\infty})$, hence ramified in $F \cap  K(\mu_{\ell^\infty})$.  Thus, by \eqref{fancyfielddiagramwithF}, $\mf{L}$ is ramified in $K(E_{\tors, (\ell)})$.  By Theorem \ref{neronoggshafarevichthm}, we find that $\ell \mid N_{K/\mbq}(\gD_E)$, finishing the proof.

\subsection{Proof of Lemma \ref{ellequals2lemma}}

The proof of Lemma \ref{ellequals2lemma} will make use of the following sub-lemma.
\begin{lemma} \label{ramificationlemma}
Let $K$ be a number field for which $2 \nmid \gD_K$ and let $\mf{p} \subseteq \mc{O}_K$ be a prime ideal lying over $2$.  Let $\ga \in \mc{O}_K - \{ 0 \}$ be any element for which $\mf{p} \nmid \ga \mc{O}_K$.  Then $2\ga$ is not a square in $K^\times$, so the field $K(\sqrt{2\ga})$ is a quadratic extension of $K$.  Furthermore, $\mf{p}$ ramifies in $K(\sqrt{2\ga})$. 
\end{lemma}
\begin{proof}
Let $v_{\mf{p}}$ denote the $\mf{p}$-adic valuation on $K$, normalized so that $v_{\mf{p}}(K^\times) = \mbz$.  Note that, since by assumption $2$ is unramified in $K$ and $v_{\mf{p}}(\ga) = 0$, we have
\begin{equation} \label{valuationof2alpha}
v_{\mf{p}}(2\ga) = v_{\mf{p}}(2) + v_{\mf{p}}(\ga) = 1,
\end{equation}
and so in particular $2\ga$ cannot be a square in $K^\times$, as asserted.  Next, let $L := K(\sqrt{2\ga})$, fix any prime $\mf{P} \subseteq \mc{O}_L$ lying over $\mf{p}$ and let $v_\mf{P}$ be the $\mf{P}$-adic valuation on $L$, normalized so that it extends $v_{\mf{p}}$ on $K$.  By \eqref{valuationof2alpha}, we then have
\[
v_{\mf{P}}\left( (2\ga)^{1/2} \right) = \frac{1}{2} v_{\mf{p}}(2\ga) = \frac{1}{2}.
\] 
It follows that $L$ is ramified over $K$ at $\mf{p}$, as asserted.
\end{proof}

We now proceed with the proof of Lemma \ref{ellequals2lemma}.  Since we are assuming that $4$ divides $r'$, by Lemma \ref{goursatlemma} we may write $G(r')$ as a fibered product:
\begin{equation} \label{fiberedproductat2}
G(r') = G(4) \times_{\psi} G(r'_{(2)}).
\end{equation}
\noindent \emph{Case:  $\GL_2(\mbz/4\mbz) \times \{ 1_{r'(2)} \} \not \subseteq G(r')$.}  In this case, either $G(4) \neq \GL_2(\mbz/4\mbz)$ or $G(4) = \GL_2(\mbz/4\mbz)$ and the common quotient $\psi_2(G(4)) = \psi_{(2)}(G(r'_{(2)}))$ in \eqref{fiberedproductat2} is nontrivial.
If $G(4) \neq \GL_2(\mbz/4\mbz)$, we find that $2 \leq [\GL_2(\mbz/4\mbz) : G(4) ] \leq [ \pi^{-1}(G(r'_{(2)})) : G(r')]$, and so the result of the lemma follows.  If on the other hand $G(4) = \GL_2(\mbz/4\mbz)$, then the common quotient in \eqref{fiberedproductat2} is nontrivial, and since
\[
\begin{split}
\pi^{-1}\left( G(r_{(2)}') \right) &= \GL_2(\mbz/4\mbz) \times G(r_{(2)}'), \\
G\left( r' \right) &= \GL_2(\mbz/4\mbz) \times_{\psi} G(r_{(2)}'),
\end{split}
\]
we thus have $2 \leq [ \pi^{-1}(G(r'_{(2)})) : G(r')]$, proving the lemma in this sub-case as well.

\noindent \emph{Case:  $\GL_2(\mbz/4\mbz) \times \{ 1_{r'(2)} \} \subseteq G(r')$.}   In this case, \eqref{fiberedproductat2} is a full product:
\begin{equation} \label{fullproductobservation}
G(r') = \GL_2(\mbz/4\mbz) \times G(r'_{(2)}).
\end{equation}
By Lemma \ref{usefulSL2elladicversion}, either $G(8)$ is an index 2 subgroup of $\GL_2(\mbz/8\mbz)$ that surjects onto $\GL_2(\mbz/4\mbz)$, or $G_2 = \GL_2(\mbz_2)$.  Let us treat the former subcase first.  A computer search reveals that there are exactly $4$ index $2$ subgroups of $\GL_2(\mbz/8\mbz)$ that map surjectively onto $\GL_2(\mbz/4\mbz)$, namely
\[
\ker(\chi_8), \; \ker(\chi_8\chi_4), \; \ker(\chi_8\ve), \ker(\chi_8\chi_4\ve),
\]
where $\chi_8 : \GL_2(\mbz/8\mbz) \rightarrow \{ \pm 1 \}$ (resp. $\chi_4 : \GL_2(\mbz/8\mbz) \rightarrow \{ \pm 1 \}$) denotes the Kronecker symbol associated to the quadratic field $\mbq(\sqrt{2})$ (resp. to $\mbq(\sqrt{-1})$), precomposed with the determinant map, and $\ve :\GL_2(\mbz/8\mbz) \rightarrow \GL_2(\mbz/2\mbz) \rightarrow \{ \pm 1 \}$ denotes the unique non-trivial character of order $2$ on $\GL_2(\mbz/2\mbz)$, precomposed with reduction modulo $2$.  We have
\begin{equation} \label{quadraticfields}
\begin{split}
&K(E[8])^{\ker(\chi_8)} = K(\sqrt{2}), \quad \quad \quad K(E[8])^{\ker(\chi_8\ve)} = K(\sqrt{2\gD_E}), \\
&K(E[8])^{\ker(\chi_8\chi_4)} = K(\sqrt{-2}), \; \quad K(E[8])^{\ker(\chi_8\chi_4\ve)} = K(\sqrt{-2\gD_E}).
\end{split}
\end{equation}
Here, by $K(\sqrt{\pm 2\gD_E})$ we mean the quadratic field $K(\sqrt{\pm 2 \gD(E_{\weier})})$, where $E_{\weier}$ is any fixed Weierstrass model of $E$ and $\gD(E_{\weier}) \in K^\times$ denotes its discriminant (note that although $\gD(E_{\weier})$ depends on the choice of $E_{\weier}$, the quadratic field $K(\sqrt{\pm 2 \gD(E_{\weier})})$ depends only on $E$).  By \eqref{quadraticfields}, we thus have
\[
G(8) = \ker(\chi_8) \; \Longrightarrow \; \sqrt{2} \in K \quad \text{ and } \quad G(8) = \ker(\chi_8\chi_4) \; \Longrightarrow \; \sqrt{-2} \in K,
\]
either of which imply that $2 \mid \gD_K$.  On the other hand, for any Weierstrass model $E_{\weier}$ of $E$, we have 
\begin{equation} \label{implicationswithepsilon}
\begin{split}
G(8) = \ker(\chi_8\ve) \; &\Longrightarrow \; \sqrt{2\gD(E_{\weier})} \in K, \\
G(8) = \ker(\chi_8\chi_4\ve) \; &\Longrightarrow \; \sqrt{-2\gD(E_{\weier})} \in K.
\end{split}
\end{equation}
Let us suppose for the sake of contradiction that 
\begin{equation} \label{2doesnotdivide}
2 \nmid \gD_K N_{K/\mbq}(\gD_E).  
\end{equation}
Fix a prime ideal $\mf{p} \subseteq \mc{O}_K$ lying over $2$.  By \eqref{2doesnotdivide}, we must have $\mf{p} \nmid \gD_E$, and we may thus find a Weierstrass model $E_{\weier}$ of $E$ satisfying
$
\mf{p} \nmid \gD(E_{\weier})
$.
Applying Lemma \ref{ramificationlemma} with $\ga = \pm \gD(E_{\weier})$, we see that $\sqrt{\pm2\gD(E_{\weier})} \notin K$, contradicting \eqref{implicationswithepsilon}.  Thus, we must have $2 \mid \gD_K N_{K/\mbq}(\gD_E)$ whenever $G(8)$ has index 2 in $\GL_2(\mbz/8\mbz)$.

We now treat the second subcase, in which $G_2 = \GL_2(\mbz_2)$.  We evidently must have a non-trivial common quotient in
\begin{equation*} 
G_{r'} \simeq \GL_2(\mbz_2) \times_{\psi} G_{r'_{(2)}}.
\end{equation*}
(If this fibered product were over a trivial quotient, then $2$ would not divide $m_G$.)  We note that any non-trivial finite quotient of $\GL_2(\mbz_2)$ must have order divisible by $2$ and that $\ker(G_{r'_{(2)}} \rightarrow G(r'_{(2)}))$ is a profinite group whose finite quotients each have order coprime with $2$.  It follows that the image of $G$ under $\id_{2} \times \pi_{(r'_{(2)})^\infty,r'_{(2)}}$ has the form
\begin{equation} \label{shortfiberedproductat2}
\GL_2(\mbz_2) \times_\psi G(r'_{(2)}),
\end{equation}
a fibered product with a common quotient of order divisible by $2$ (and hence non-trivial).  Consider the subgroup $N := \ker \psi_2 \subseteq \GL_2(\mbz_2)$ where $\psi = (\psi_2,\psi_{(2)})$ in \eqref{shortfiberedproductat2}.  The assumption $\GL_2(\mbz/4\mbz) \times \{ 1_{r'(2)} \} \subseteq G(r')$ then implies that $N(4) = \GL_2(\mbz/4\mbz)$ (otherwise the mod $r'$ image of \eqref{shortfiberedproductat2} would have a nontrivial fibering between $G(4)$ and $G(r'_{(2)})$, contradicting \eqref{fullproductobservation}).  By Lemma \ref{usefulSL2elladicversion}, we find that $[\GL_2(\mbz/8\mbz) : N(8) ] = 2$.  By the same computation as mentioned in the previous subcase, we have
\[
N(8) \in \{ \ker(\chi_8), \; \ker(\chi_8\chi_4), \; \ker(\chi_8\ve), \ker(\chi_8\chi_4\ve) \},
\]
and it follows by \eqref{quadraticfields} and Galois theory that one of the fields $K(\sqrt{2})$, $K(\sqrt{-2})$, $K(\sqrt{2\gD_E})$, or $K(\sqrt{-2\gD_E})$ must be contained in $K(E_{\tors, (2)})$.  By Lemma \ref{ramificationlemma} and Theorem \ref{neronoggshafarevichthm}, it follows that, if $2 \nmid \gD_K$ then $2$ divides $N_{K/\mbq}(\gD_E)$.  This finishes the proof of Lemma \ref{ellequals2lemma}. 

\section{Appendix:  MAGMA code for computations}

The computations referenced throughout the paper were undertaken using the computational package MAGMA (see \cite{MAGMA}); below we give scripts that the interested reader can run therein to reproduce those results. \\

----------------------------------------------------

\noindent {\tt{> IsConjugateToSubgroup := function(G1,G2,m)}} \\
\noindent {\tt{function>       answer := false;}} \\
\noindent {\tt{function>       for H2 in Subgroups(G2 : OrderEqual := \#G1) do}} \\
\noindent {\tt{function|for>           if IsConjugate(GL(2,Integers(m)),G1,H2`subgroup) then}} \\
\noindent {\tt{function|for|if>                        answer := true;}} \\
\noindent {\tt{function|for|if>                        break H2;}} \\
\noindent {\tt{function|for|if>                end if;}} \\
\noindent {\tt{function|for>   end for;}} \\
\noindent {\tt{function>       return answer;}} \\
\noindent {\tt{function> end function;}}\\

---------------------------------------------------- \\

\noindent {\tt{> G3 := GL(2,Integers(3));}} \\
\noindent {\tt{> C3 := CommutatorSubgroup(G3,G3);}} \\
\noindent {\tt{> C3 eq SL(2,Integers(3));}} \\

---------------------------------------------------- \\

\noindent {\tt{> L9 := [];}} \\
\noindent {\tt{> G9 := GL(2,Integers(9));}} \\
\noindent {\tt{> for H in Subgroups(G9) do}} \\
\noindent {\tt{for> H3 := MatrixGroup<2,Integers(3) | Generators(H`subgroup)>;}} \\
\noindent {\tt{for> if H3 eq G3 and not SL(2,Integers(9)) subset H`subgroup then}} \\
\noindent {\tt{for|if> Append($\sim$L9,H`subgroup);}} \\
\noindent {\tt{for|if> end if;}} \\
\noindent {\tt{for> end for;}} \\
\noindent {\tt{> IsConjugateToSubgroup(L9[1],L9[2],9);}} \\
\noindent {\tt{> \#GL(2,Integers(9))/\#L9[2];}} \\

----------------------------------------------------- \\

\noindent {\tt{> G5 := GL(2,Integers(5));}} \\
\noindent {\tt{> L5 := [];}} \\
\noindent {\tt{> for N in NormalSubgroups(G5) do}} \\
\noindent {\tt{for> if IsDivisibleBy(Floor(\#GL(2,Integers(5))/\#N`subgroup),60) then}} \\
\noindent {\tt{for|if> Append($\sim$L5,N`subgroup);}} \\
\noindent {\tt{for|if> end if;}} \\
\noindent {\tt{for> end for;}} \\
\noindent {\tt{> L5;}} \\

----------------------------------------------------- \\

\noindent {\tt{> L9p := [];}} \\
\noindent {\tt{> for H in NormalSubgroups(G9) do}} \\
\noindent {\tt{for> Hext := MatrixGroup<2,Integers(9) | Generators(SL(2,Integers(9))), Generators(H`subgroup)>;}} \\
\noindent {\tt{for> if Hext eq GL(2,Integers(9)) then}} \\
\noindent {\tt{for|if> Append($\sim$L9p,Hext);}} \\
\noindent {\tt{for|if> end if;}} \\
\noindent {\tt{for> end for;}} \\
\noindent {\tt{> \#L9p;}} \\
\noindent {\tt{> L9p[1] eq GL(2,Integers(9));}} \\

----------------------------------------------------- \\

\noindent {\tt{> L8 := [];}} \\
\noindent {\tt{> for G in Subgroups(G8 : IndexEqual := 2) do}} \\
\noindent {\tt{for> G4 := MatrixGroup<2,Integers(4) | Generators(G`subgroup)>;}} \\
\noindent {\tt{for> if G4 eq GL(2,Integers(4)) then}} \\
\noindent {\tt{for|if> Append($\sim$L8,G`subgroup);}} \\
\noindent {\tt{for|if> end if;}} \\
\noindent {\tt{for> end for;}} \\
\noindent {\tt{> L8;}} \\


\begin{thebibliography}{99}

\bibitem{biluparent} Y. Bilu and P. Parent, \emph{Serre's uniformity problem in the split Cartan case}, Annals of Mathematics, \textbf{173}, no. 1 (2011), 569--584.

\bibitem{biluparentrebolledo} Y. Bilu, P. Parent and M. Rebolledo, \emph{Rational points on $X_0^+(p^r)$}, Annales de l'Institut Fourier, \textbf{63}, no. 3 (2013), 957--984.


\bibitem{awstitchmarsh} R. Bell, C. Blakestad,  A.C. Cojocaru, A. Cowan, N. Jones, V. Matei, G. Smith and I. Vogt, \emph{Constants in Titchmarsh divisor problems for elliptic curves}, preprint.  Available at {\tt{https://arxiv.org/abs/1706.03422}}

\bibitem{MAGMA} W. Bosma, J. Cannon and C. Playoust, \emph{The {M}agma algebra system. {I}. {T}he user language}, J. Symbolic Comput. \textbf{24}, no. 3-4, (1997), 235--265.

\bibitem{bourdon} A. Bourdon, O. Ejder, Y. Liu, F. Odumodu and B. Viray, \emph{On the level of modular curves that give rise to isolated $j$-invariants}, Adv. Math. \textbf{357} (2019), 106824, 33 pp.

\bibitem{braujones} J. Brau and N. Jones, \emph{Elliptic curves with $2$-torsion contained in the $3$-torsion field}, Proc. Amer. Math. Soc. \textbf{144} (2016), 925--936.


\bibitem{cojocaru} A.C. Cojocaru, \emph{On the surjectivity of the Galois representations associated to non-CM elliptic curves}, Canad. Math. Bull. \textbf{48}, no. 1, (2005) 16--31.

\bibitem{daniels} H. Daniels, \emph{An infinite family of Serre curves}, J. Number Theory \textbf{155} (2015), 226--247.

\bibitem{danielsgonzalez} H. Daniels and E. Gonzalez-Jimenez, \emph{Serre's constant of elliptic curves over the rationals}, to appear in Exp. Math.  Available at {\tt{https://arxiv.org/abs/1812.04133}}



\bibitem{dokchitser} T. Dokchitser and V. Dokchitser, \emph{Surjectivity of mod $2^n$ representations of elliptic curves}, Math. Z. \textbf{272} (2012), 961--964.

\bibitem{elkies} N. Elkies, \emph{Elliptic curves with $3$-adic Galois representation surjective mod $3$ but not mod $9$}, preprint (2006).

\bibitem{huppert} B. Huppert, \emph{Endliche Gruppen I}, Die Grundlehren der mathematischen Wissenschaften, \textbf{134}, Springer-Verlag, Berlin Heidelberg (1967).

\bibitem{jones1} N. Jones, \emph{A bound for the torsion conductor of an elliptic curve}, Proc. Amer. Math. Soc. \textbf{137} (2009), 37--43.

\bibitem{lozanorobledo} A. Lozano-Robledo,  \emph{On the field of definition of p-torsion points on elliptic curves over the rationals}, Math. Annalen \textbf{357}, no. 1 (2013), 279--305.





\bibitem{mazur} B. Mazur, \emph{Rational isogenies of prime degree}, Invent. Math. \textbf{44}, no. 2 (1978), 129--162.

\bibitem{merel} L. Merel, \emph{Bornes pour la torsion des courbes elliptiques sure les corps de nombres}, Invent. Math. \textbf{24}, no. 1 (1996), 437--449.

\bibitem{morrow} J. Morrow, \emph{Composite images of Galois for elliptic curves over $\mbq$ and Entanglement fields}, Math. Comp. \textbf{88} (2019), 2389--2421.

\bibitem{ogg} A.P. Ogg, \emph{Elliptic curves and wild ramification}, Amer. J. Math. \textbf{89}, no. 1 (1967), 1--21.

\bibitem{ribet} K. Ribet, \emph{Galois action on division points of Abelian varieties with real multiplications}, Amer. J. Math. \textbf{98}, no. 3 (1976), 751--804.

\bibitem{serre} J-P. Serre,  \emph{Propri\'{e}t\'{e}s galoisiennes des points d'ordre fini des courbes elliptiques}, Invent. Math. \textbf{15} (1972), 259--331.

\bibitem{serre2} 
J.-P. Serre, 
\emph{Abelian $\ell$-adic representations and elliptic curves}, 
Benjamin, New York-Amsterdam, 1968.

\bibitem{silverman} J. Silverman, \emph{The arithmetic of elliptic curves}, Springer, New York, 1986.

\bibitem{sutherlandzywina} A. Sutherland and D. Zywina, \emph{Modular curves of prime-power level with infinitely many rational points}, Algebra \& Number Theory \textbf{11}, no. 5 (2017), 1199--1229.

\bibitem{weil} A. Weil, \emph{Sur les fonctions algebriques \`{a} corps de constantes finis}, C. R. Acad. Sci. Paris \textbf{210} (1940) 592--594.

\bibitem{zywina} D. Zywina, \emph{Elliptic curves with maximal Galois action on their torsion points}, 
Bull. Lond. Math. Soc. \textbf{42} (2010), 811--826.

\bibitem{zywina2} D. Zywina, \emph{On the possible images of the mod $\ell$ representations associated to elliptic curves over $\mbq$}, preprint.  Available at {\tt{https://arxiv.org/abs/1508.07660}}

\bibitem{zywinaindex} D. Zywina, \emph{Possible indices for the Galois image of elliptic curves over $\mbq$}, preprint.  Available at {\tt{https://arxiv.org/abs/1508.07663}}

\end{thebibliography}
\end{document}